\documentclass[reqno,11pt]{amsart}
\usepackage{amssymb}

\usepackage{amsmath}
\usepackage{graphicx}
\usepackage{color}
\usepackage{slashbox}

\numberwithin{equation}{section}
\newtheorem{theorem}{Theorem}[section]
\newtheorem{proposition}{Proposition}[section]
\newtheorem{lemma}{Lemma}[section]
\newtheorem{corollary}{Corollary}[section]

\theoremstyle{definition}
\newtheorem{definition}{Definition}[section]

\newtheorem{remark}{Remark}[section]
\newtheorem{example}{Example}[section]

\input{tcilatex}

\begin{document}
\date{December 19, 2010}
\title{Best approximation in max-plus semimodules}

\begin{abstract}
We establish new results concerning projectors on max-plus spaces, as well
as separating half-spaces, and derive an explicit formula for the distance
in Hilbert's projective metric between a point and a half-space over the
max-plus semiring, as well as explicit descriptions of the set of
minimizers. As a consequence, we obtain a cyclic projection type algorithm
to solve systems of max-plus linear inequalities.
\end{abstract}

\author{Marianne Akian}
\address{Marianne Akian, INRIA Saclay--\^Ile-de-France and CMAP, \'Ecole
Polytechnique. Address: CMAP, \'Ecole Polytechnique, Route de Saclay, 91128
Palaiseau Cedex, France.}
\email{Marianne.Akian@inria.fr}
\author{St\'{e}phane Gaubert}
\address{St\'ephane Gaubert, INRIA Saclay--\^Ile-de-France and CMAP, \'Ecole
Polytechnique. Address: CMAP, \'Ecole Polytechnique, Route de Saclay, 91128
Palaiseau Cedex, France.}
\email{Stephane.Gaubert@inria.fr}
\author{Viorel Ni\c{t}ic\u{a}}
\address{Viorel Ni\c{t}ic\u{a}, Department of Mathematics, West Chester University,
PA 19383, U. S. A. and Institute of Mathematics, P. O. Box 1-764, Bucharest,
Romania}
\email{vnitica@wcupa.edu}
\author{Ivan Singer}
\address{Ivan Singer, Institute of Mathematics, P. O. Box 1-764, Bucharest, Romania}
\email{ivan.singer@imar.ro}
\thanks{The four authors were partially supported by a LEA ``Math Mode'' grant for
2009-2010. The two first authors were also partially supported by the joint
RFBR-CNRS grant 05-01-02807 and by a MSRI Research membership for the Fall
2009 Semester on Tropical Geometry. The second author was also partially
supported by the Arpege programme of the French National Agency of Research
(ANR), project ``ASOPT'', number ANR-08-SEGI-005. Some results of this paper
have been presented at the Montr\'{e}al Workshop on Tropical and Idempotent
Mathematics, CRM/GERAD, June 29-July 3rd 2009.}
\maketitle

\section{Introduction}

Let $\mathbb{R}_{\max }$ denote the so-called \emph{max-plus algebra}\/,
which is the semiring composed of the set $\mathbb{R}\cup \{-\infty \}$
endowed with the maximization operation as addition $\mu \oplus \nu :=\max
(\mu ,\nu )$, the usual addition as multiplication $\mu \otimes \nu :=\mu
+\nu $ (also for $\mu =\nu =-\infty )$, and the neutral elements $-\infty $
and $0$ for addition $\oplus $ and multiplication $\otimes $ respectively.
We shall often denote the multiplication of $\mathbb{R}_{\max}$ by
concatenation (except when the omission of the symbol $\otimes$ leads to an
ambiguity).

The space $\mathbb{R}_{\max }^{n}$ of $n$-dimensional vectors, endowed
naturally with the pointwise addition (denoted also by $\oplus $) and the
multiplication of a vector by a scalar (denoted below by concatenation, with
the scalar on the right), is a semimodule (the analogue of a module) over $%
\mathbb{R}_{\max }$. It is also endowed with the following operation $%
\backslash $ which comes from the residuation of the map that multiplies a
scalar by a given vector (see Section~\ref{notations}): 
\begin{equation}
{x}\backslash {y}:=\sup \{\lambda \in \mathbb{R}_{\max }\mid x\lambda \leq
y\}\enspace,  \label{resid-vec-fin}
\end{equation}
where the order $\leq $ in~(\ref{resid-vec-fin}) is the usual partial order.

The most natural ``distance''~\cite{CGQ,Ben-01,GK,JSY} on the space $\mathbb{%
R}_{\max }^{n}$ is the (additive analogue of) Hilbert's projective distance $%
d$, which can be defined by 
\begin{equation}
d(x,y):=((x\backslash y)\otimes (y\backslash x))^{-},  \label{projdist}
\end{equation}
where the superscript $-$ means taking the usual opposite, that is, 
\begin{equation}
\lambda ^{-}:=-\lambda \quad \quad \forall \lambda \in \overline{\mathbb{R}}%
:=\mathbb{R}\cup \{-\infty ,+\infty \};  \label{oppo}
\end{equation}
note that here $d(x,y)\in \overline{\mathbb{R}}$. When the vectors $x$ and $%
y $ have only finite entries, 
\begin{equation*}
d(x,y)=\max_{i,j\in[n]}(x_i-y_i+y_j-x_j) \enspace,
\end{equation*}
where $[n]:=\{1,\dots,n\}$.

The same definition can be used on any residuated idempotent semimodule, and
it generalizes there the usual Hilbert projective metric considered on cones
of Banach spaces~\cite{birkhoff}: if $u,v$ are two vectors in the interior
of a closed convex pointed cone $C$ in such a space, the Hilbert projective
metric is classically defined by 
\begin{equation}
\func{Hilb}(u,v)=\min \log \{\frac{\mu }{\lambda }|\lambda >0,\mu >0,\quad
\lambda u\leq v\leq \mu u\}\enspace,  \label{e-birkhoff}
\end{equation}
where $u\leq v$ means that $v-u\in C$. When $\mathbb{R}^{n}$ is thought of
as the image of the interior of the standard positive cone by the map which
takes the logarithm entrywise, so that $x_{i}=\log u_{i}$ and $y_{i}=\log
v_{i}$, we get $d(x,y)=\func{Hilb}(u,v)$ (see~\cite[Section~3.3]{CGQ}).

If one avoids vectors with only infinite entries, then $d$ satisfies all the
properties of a projective distance, except that it may take infinite values
(see Section~\ref{notations}).

If $V$ is a subset of $\mathbb{R}_{\max }^{n}$, and $x\in \mathbb{R}_{\max
}^{n}$, one defines as for a usual distance: 
\begin{equation}
d(x,V):=\inf_{v\in V}d(x,v),  \label{D3}
\end{equation}
and we define an \emph{element of best approximation}\/, or a \emph{best
approximation}, of $x$ in $V,$ or a \emph{nearest point} to $x$ in $V,$ as
an element $v_{0}$ of $V$ such that 
\begin{equation}
d(x,v_{0})=d(x,V).  \label{D4}
\end{equation}

In the present paper we shall study the best approximation for Hilbert's
projective metric in \emph{\emph{b}-complete subsemimodules} of $\mathbb{R}%
_{\max }^{n}.$ We recall that any semimodule $V$ over $\mathbb{R}_{\max }$
is an idempotent monoid for its additive law, and is thus ``naturally''
ordered by the relation $\leq $ defined by 
\begin{equation}
x\leq y\Leftrightarrow x\oplus y=y,  \label{e-def-order}
\end{equation}
which is such that the supremum coincides with the addition $\oplus $ of the
semimodule. It is said to be \emph{\emph{b}-complete} if any subset of $V$
bounded from above has a supremum in $V$ and if the scalar multiplication
distributes over all such infinite sums (see Litvinov, Maslov and Shpiz~\cite
{LMS}). In particular, $\mathbb{R}_{\max }^{n}$ is a \emph{b}-complete
semimodule over $\mathbb{R}_{\max }$, and its natural order is the usual
partial order. A subsemimodule $V$ of $\mathbb{R}_{\max }^{n}$ is a \emph{%
\emph{b}-complete subsemimodule} of $\mathbb{R}_{\max }^{n}$ if the supremum
of any subset of $V$ bounded from above belongs to $V$.

Let us also recall that for a \emph{b}-complete subsemimodule $V$ of $%
\mathbb{R}_{\max }^{n}$ the \emph{canonical projection operator}\/ $P_{V}$
of $\mathbb{R}_{\max }^{n}$ onto $V$ is defined \cite{CGQ} by 
\begin{equation}
P_{V}(x):=\max \{v\in V|v\leq x\}, \quad \quad \forall x\in \mathbb{R}_{\max
}^{n},  \label{phx1}
\end{equation}
where max denotes a supremum which is attained (by some element of $V$).
Then (see \cite{CGQ,GK,JSY}) for any $x\in \mathbb{R}_{\max }^{n}$, $%
P_{V}(x) $ is a best approximation of $x$ in $V$ (such a best approximation
is not necessarily unique), that is, 
\begin{equation}
d(x,P_{V}(x))=d(x,V).  \label{phx21}
\end{equation}

Some of our results are inspired by -and bear some analogy with- those known
from the theory of best approximation in normed linear spaces by elements of
linear subspaces (see e.g.\ \cite{S}), reformulated in terms of the
``semi-scalar product'' (see e.g.\ \cite{L}). These analogies have led us
even to the discovery of some new properties of the canonical projections
onto semimodules (see e.g.\ Theorem~\ref{th1} and Corollaries~\ref{label2}, 
\ref{coromindist}, \ref{chop}).

The structure of the paper is as follows. \textbf{\ }

In the preliminary Section~\ref{notations} we give some notations, concepts and
facts that will be used in the sequel, concerning residuation for scalars,
vectors and matrices and its connections with the additions $+$ and $%
+^{\prime }$ on $\overline{\mathbb{R}},$ and the Hilbert projective distance 
$d$ and anti-distance $\delta $ on a complete semimodule $X,$ with special
emphasis on the particular cases $X=\overline{\mathbb{R}}_{\max }^{n}$ and $%
X=\mathbb{R}_{\max }^{n}.$

In Section~\ref{s03} we introduce the support, upper support and lower support and
the ``part'' of an element $x\in \overline{\mathbb{R}}_{\max }^{n}$ and we
show that with the aid of these concepts one can reduce the study of best
approximation of the elements $x\in \mathbb{R}_{\max }^{n}$ by the elements
of a \emph{b}-complete subsemimodule $V$ of $\mathbb{R}_{\max }^{n}$ to the
case where $x\in \mathbb{R}^{n}$ and $V\subset \mathbb{R}^{n}\cup \{-\infty
\}$, where $-\infty $ denotes the vector of $\mathbb{R}_{\max }^{n}$ with
all its entries equal to $-\infty $.

In Section~\ref{s04}, using the known fact \cite{CGQ,GK} that for every \emph{b}%
-complete subsemimodule $V$ of $\mathbb{R}_{\max }^{n}$ and every outside
point $x$ there exists a ``universal separating half-space'' $H=H_{V,x},$
defined with the aid of $P_{V}(x),$ satisfying $V\subseteq H_{V,x}$ and $%
x\in \mathbb{R}^{n}\backslash H_{V,x},$ we show that the problem of best
approximation of $x$ by elements of a \emph{b}-complete subsemimodule $V$ of 
$\mathbb{R}_{\max }^{n}$ can be reduced to the problem of best approximation
of $x$ by elements of a closed half-space $H$ of $\mathbb{R}_{\max }^{n}.$
To this end we prove the following properties of $H$: for each $x\in \mathbb{%
R}^{n}\backslash V$ we have $P_{V}(x)=P_{H}(x)$ and $d(x,V)=d(x,H).$ As in 
\cite{CGQ}, for more transparency we prove first corresponding results for
``complete subsemimodules'' of $\overline{\mathbb{R}}_{\max }^{n}$ and
separation by ``complete half-spaces'' of $\overline{\mathbb{R}}_{\max
}^{n}, $ from which we deduce the results on $\mathbb{R}_{\max }^{n}.$

In Section~\ref{s05} we prove for a closed half-space $H$ of $\mathbb{R}_{\max }^{n}$
and an outside point $x\in \mathbb{R}^{n}\backslash H$ a formula for the
distance $d(x,H),$ and we obtain a formula for the canonical projection $%
P_{H}(x)$ of $x$ onto $H.$

In Section~\ref{s06} we show that every closed half-space of $\mathbb{R}_{\max }^{n}$
admits a canonical representation with the aid of coefficients with disjoint
supports, and we particularize this result to obtain the canonical form of
the universal separating closed half-space of a \emph{b}-complete
subsemimodule $V$ of $\mathbb{R}_{\max }^{n}$ from a point $x\notin V.$ The
latter canonical form shows that when the canonical projection of $x$ onto $%
V $ is finite, the universal separating closed half-spaces always have
``finite apex''.

In Section~\ref{s07}, using the results of Section~\ref{s06}, we give characterizations of
the elements of best approximation by arbitrary half-spaces (not necessarily
with finite apex) for an element $x\in \mathbb{R}_{\max }^{n}$. At the end
of the section we also give geometric interpretations in simple particular
cases.

Finally, in Section~\ref{s08}, as an application of the main distance formula of
Section~\ref{s05}, we obtain a new algorithm to solve systems of max-plus linear
inequalities $Ax\geq Bx$, where $A,B$ are $p\times n$ matrices. This
algorithm uses the technique of cyclic projectors \cite{GS}; it may be
thought of as a max-plus analogue of the Gauss-Seidel algorithm, and it is
shown to be faster than the earlier alternated projection algorithm of~\cite
{CGB}, although it remains only pseudo-polynomial.

Let us mention that the results on $X=\overline{\mathbb{R}}_{\max }^{n}$ of
this paper can be extended to more general assumptions on a complete
semimodule $X.$ To this end, one needs to extend the concept of ``opposite'' 
$\lambda ^{-}$ of (\ref{oppo}). A rather complete theory of an extension of
the ``opposite'' is developed in \cite{CGQ}, but we shall not pursue here
that level of generality.

\section{Notations and preliminaries}

\label{notations}

\subsection{Residuation}

As mentioned above, we denote by $\mathbb{R}_{\max }$ the semiring composed
of the set $\mathbb{R}\cup \{-\infty \}$ endowed with the maximization
operation as addition $\mu \oplus \nu :=\max (\mu ,\nu )$, the usual
addition as multiplication $\mu \otimes \nu :=\mu +\nu $ (also for $\mu =\nu
=-\infty )$, and the neutral elements $-\infty $ and $0$ for addition $%
\oplus $ and multiplication $\otimes $ respectively. Furthermore, we shall
denote by $\overline{\mathbb{R}}_{\max }$ the so-called \emph{complete
max-plus algebra}\/, which is the semiring composed of the set $\overline{%
\mathbb{R}}:=\mathbb{R}\cup \{-\infty ,+\infty \}$ endowed with the
maximization operation as addition, that is, 
\begin{equation*}
\mu \oplus \nu :=\max (\mu ,\nu ),
\end{equation*}
and with the extension to $\overline{\mathbb{R}}$ of the usual addition $+$
of $\mathbb{R}\cup \{-\infty \}$ as multiplication $\mu \otimes \nu =\mu
+\nu ,$ by the convention 
\begin{equation}
a+(+\infty )=(+\infty )+a=\left\{ 
\begin{array}{l}
+\infty \quad \text{if }a\in \mathbb{R}\cup \{+\infty \} \\ 
-\infty \quad \text{if }a=-\infty .
\end{array}
\right.  \label{conven}
\end{equation}

Throughout this paper we shall consider the space $\mathbb{R}_{\max }^{n}$
(respectively, $\overline{\mathbb{R}}_{\max }^{n})$ of all $n$-dimensional
column vectors $x=(x_{1},\dots,x_{n})^{T},$ where $x_{1},\dots,x_{n}$ belong
to $\mathbb{R}_{\max }$ (respectively, $\overline{\mathbb{R}}_{\max })$ and
the superscript $\cdot^T$ denotes the transposition operation, endowed
naturally with the pointwise addition (denoted by $\oplus $) and
multiplication by a scalar, that we shall denote by a concatenation on the
right. This is a semimodule over $\mathbb{R}_{\max }$ (respectively, $%
\overline{\mathbb{R}}_{\max })$. We shall denote such column vectors, or
equivalently, $n\times 1$ matrices, by the letters $x,y,z,u,h$,\dots We
shall also consider matrices over $\mathbb{R}_{\max }$ and $\overline{%
\mathbb{R}}_{\max },$ denoted by capital letters $A,B,\dots$ and employ the
usual concatenation notation for product of matrices, as well as for the
multiplication of an element of $\mathbb{R}_{\max }^{n}$ (or $\overline{%
\mathbb{R}}_{\max }^{n})$ by a scalar, that we shall put on the right (as if
scalars were one dimensional square matrices). So if $x=(x_{1},\ldots
,x_{n})^{T}\in \mathbb{R}_{\max }^{n}$ (or $\overline{\mathbb{R}}_{\max
}^{n}),$ and $\lambda \in \mathbb{R}_{\max }$ (respectively $\overline{%
\mathbb{R}}_{\max })$, then $x\lambda $ is the vector $(x_{1}+\lambda
,\ldots ,x_{n}+\lambda )^{T}$ (the notation $x+\lambda $ is also used in the
literature).

As in usual algebra, any max-plus linear operator $\phi $ from $\mathbb{R}%
_{\max }^{n}$ to $\mathbb{R}_{\max }^{m}$ (respectively $\overline{\mathbb{R}%
}_{\max }^{n}$ to $\overline{\mathbb{R}}_{\max }^{m}$), i.e., satisfying $%
\phi (x\oplus y)=\phi (x)\oplus \phi (y)$ for all $x,y\in \mathbb{R}_{\max
}^{n}$ (respectively $\overline{\mathbb{R}}_{\max }^{n})$ and $\phi
(x\lambda )=\phi (x)\lambda $ for all $x\in \mathbb{R}_{\max }^{n}\;(%
\overline{\mathbb{R}}_{\max }^{n})$ and $\lambda \in \mathbb{R}_{\max }$ $(%
\overline{\mathbb{R}}_{\max })$ can be represented by (and identified to) a $%
m\times n$ matrix $A=(A_{ij})_{i\in \lbrack m],j\in \lbrack n]}$ over $%
\mathbb{R}_{\max }$ (respectively $\overline{\mathbb{R}}_{\max }$), with $%
\phi (x)=Ax$, that is $\phi (x)_{i}=\max_{j\in \lbrack n]}(A_{ij}+x_{j})$
for $i\in [m]$ (see~\cite{BCOQ}). In particular, when $m=1,$ the dual space $%
(\mathbb{R}_{\max }^{n})^{\ast }$ (respectively $(\overline{\mathbb{R}}%
_{\max }^{n})^{\ast })$ of all max-plus linear forms over $\mathbb{R}_{\max
}^{n}$ (respectively $\overline{\mathbb{R}}_{\max }^{n}$), that is, of all
max-plus linear functions $(\mathbb{R}_{\max }^{n})^{\ast }\rightarrow 
\mathbb{R}_{\max }$ (respectively $(\overline{\mathbb{R}}_{\max }^{n})^{\ast
}\rightarrow \overline{\mathbb{R}}_{\max })$ is isomorphic, and shall be
identified, with the space of all $n$-dimensional row vectors, or
equivalently, $1\times n$ matrices, having their entries in $\mathbb{R}%
_{\max }$ (respectively, $\overline{\mathbb{R}}_{\max })$, which we shall
denote by $a=(a_{1},\dots,a_{n}),b$, \dots

Spaces of scalars, vectors and matrices over $\mathbb{R}_{\max }$ ($%
\overline{\mathbb{R}}_{\max })$ are idempotent monoids with respect to
addition and their ``natural order'' for which the supremum operation is
equivalent to the addition of the monoid, and that order coincides with the
usual partial order. They are \emph{b}-complete (complete) semimodules over $%
\mathbb{R}_{\max }$ ($\overline{\mathbb{R}}_{\max }$), in the sense that
will be recalled below. This allows one to define the residuation operation $%
A\backslash B$ for any matrices $A\in \overline{\mathbb{R}}_{\max }^{n\times
m}$ and $B\in \overline{\mathbb{R}}_{\max }^{n\times p}$ by 
\begin{equation}
{A}\backslash {B}:=\max \{C\in \overline{\mathbb{R}}_{\max }^{m\times p}\mid
AC\leq B\},  \label{def-resid}
\end{equation}
where the max means that that the supremum is attained; in particular, for
any scalars $\mu ,\nu \in \overline{\mathbb{R}}_{\max },$%
\begin{equation}
\mu \backslash \nu :=\max \{\lambda \in \overline{\mathbb{R}}_{\max }\mid
\mu \otimes \lambda \leq \nu \}.  \label{def-resid2}
\end{equation}
Since semimodules of matrices with entries in $\mathbb{R}_{\max }$ are not
complete but only \emph{b}-complete, the residuation $A\backslash B$ of
matrices $A\in \mathbb{R}_{\max }^{n\times m}$ and $B\in \mathbb{R}_{\max
}^{n\times p}$ is not necessarily in $\mathbb{R}_{\max }^{m\times p}$;
however one can replace the maximum in~the definition (\ref{def-resid}) of $%
A\backslash B$ by the supremum in $\mathbb{R}_{\max }^{m\times p}$, as in (%
\ref{resid-vec-fin}).

Let us denote by $\overline{\mathbb{R}}_{\min }$ the so-called \emph{%
complete min-plus algebra}, which is by definition the semiring composed of
the set $\overline{\mathbb{R}}$ endowed with the minimization operation as
addition $\mu \oplus ^{\prime }\nu ,$ that is, 
\begin{equation*}
\mu \oplus ^{\prime }\nu :=\min (\mu ,\nu ),
\end{equation*}
and with the extension to $\overline{\mathbb{R}}$ of the usual addition $+$
of $\mathbb{R}\cup \{+\infty \}$ as multiplication $\mu \otimes ^{\prime
}\nu =\mu +^{\prime }\nu ,$ defined by the convention opposite to~(\ref
{conven}), namely: 
\begin{equation}
a+^{\prime }(-\infty )=(-\infty )+^{\prime }a=\left\{ 
\begin{array}{l}
+\infty \quad \text{if }a=+\infty \\ 
-\infty \quad \text{if }a\in \mathbb{R}\cup \{-\infty \}.
\end{array}
\right.  \label{conven-min}
\end{equation}
The neutral elements of $\overline{\mathbb{R}}_{\min }$ are necessarily $%
+\infty $ and $0$ for addition $\oplus ^{\prime }=\min $ and multiplication $%
\otimes ^{\prime }=+^{\prime }$ respectively.

\begin{remark}
\label{rmor}\emph{a)} The above operations $\otimes =+$ and $\otimes
^{\prime }=+^{\prime }$ are nothing else than the ``lower addition'' $%
\underset{\cdot }{+}$ and ``upper addition'' $\overset{\cdot }{+}$ on $%
\overline{\mathbb{R}}$ respectively, introduced by Moreau (see e.g.~\cite
{mor}) and used extensively in convex analysis. This remark permits to
extend the well-known results about $\underset{\cdot }{+}$ and $\overset{%
\cdot }{+}$ on $\overline{\mathbb{R}}$ to the lower and upper product $%
\underset{\cdot }{\otimes }$ and $\overset{\cdot }{\otimes }$ respectively,
on any complete semifield $\mathbb{S}$, using the known rules for these
operations (see e.g.\ \cite{ak-sin}).

\emph{b)} Here we consider mainly operations of $\overline{\mathbb{R}}_{\max
}$, whereas those of $\overline{\mathbb{R}}_{\min }$ are considered as dual
ones, hence the notations $+$ and $+^{\prime }$. Such ``dual'' notations
were already used in the literature, e.g.\ in~\cite{cuning79}.
\end{remark}

We recall the following well-known rules of computation with $+$ and $%
+^{\prime }$ on $\overline{\mathbb{R}}$:

\begin{lemma}
\label{lthree}(\cite{mor}, formulas (2.1) and (2.3)). For any $\lambda ,\mu
,\nu \in \overline{\mathbb{R}}$ we have 
\begin{align}
-(\mu +^{\prime }\nu )&=-\mu +(-\nu ),  \label{sim} \\
(\lambda +^{\prime }\mu )+^{\prime }\nu &=\lambda +^{\prime }(\mu +^{\prime
}\nu ).  \notag
\end{align}
\end{lemma}

By (\ref{sim}), the semiring $\overline{\mathbb{R}}_{\min }$ can also be
defined equivalently as the image of $\overline{\mathbb{R}}_{\max }$ by the
``opposite'' map $\overline{\mathbb{R}}\rightarrow \overline{\mathbb{R}}%
,\;x\mapsto x^{-}$ with $x^{-}$ defined as in~(\ref{oppo}), which means that
the opposite map is an isomorphism of complete semirings from $\overline{%
\mathbb{R}}_{\max }$ to $\overline{\mathbb{R}}_{\min }$.

For the basic rules of computation with residuation of scalars and their
extensions to residuation of vectors and matrices see e.g.\ \cite{BCOQ,CGQ}.

Let us give now some new properties of the residuation of scalars that we
shall use later.

\begin{proposition}
\label{lone}For $\mu ,\nu \in \overline{\mathbb{R}}_{\max }$, we have 
\begin{equation}
{\mu }\backslash {\nu }=\nu +^{\prime }(-\mu ),  \label{resid-scalar-usual}
\end{equation}
with $+^{\prime }$ of \emph{(\ref{conven-min})}.
\end{proposition}

\textbf{Proof}. By Definition~(\ref{def-resid2}), we have 
\begin{equation*}
{\mu }\backslash {\nu }:=\max \{\lambda \in \overline{\mathbb{R}}_{\max
}\mid \mu \otimes \lambda \leq \nu \}\enspace,
\end{equation*}
that is, in usual notations (with the convention~(\ref{conven}) for $+),$ 
\begin{equation}
{\mu }\backslash {\nu }=\max \{\lambda \in \overline{\mathbb{R}}\mid \mu
+\lambda \leq \nu \}.  \label{miuniu2}
\end{equation}
But, by \cite{mor}, p. 119, Proposition 3(c), for any $\mu ,\nu ,\lambda \in 
\overline{\mathbb{R}}$ we have the equivalence 
\begin{equation}
\mu +\lambda \leq \nu \Leftrightarrow \lambda \leq \nu +^{\prime }(-\mu )%
\enspace,  \label{resid-scal-usual}
\end{equation}
whence, by (\ref{miuniu2}) and (\ref{resid-scal-usual}), we obtain 
\begin{equation*}
\mu \backslash \nu =\max \{\lambda \in \overline{\mathbb{R}}\mid \lambda
\leq \nu +^{\prime }(-\mu )\}=\nu +^{\prime }(-\mu )\enspace.\quad \square
\end{equation*}

\begin{remark}
For a somewhat similar result see \cite[the remark made after Example 3.2]
{DMR}.
\end{remark}

\begin{corollary}
\label{ltwo}For $\mu ,\nu \in \overline{\mathbb{R}}_{\max }$, we have 
\begin{gather}
\mu \backslash \nu \in \mathbb{R}\Leftrightarrow \mu \text{ and }\nu \in 
\mathbb{R}\enspace,  \label{resid-finite} \\
\mu \backslash \nu =+\infty \Leftrightarrow \mu =-\infty \text{ or }\nu
=+\infty \text{ (or both).}  \label{resid-pinf}
\end{gather}
\end{corollary}

\textbf{Proof.} This follows from Proposition \ref{lone} and the definition
of $+^{\prime },$ since $-\mu =+\infty $ if and only if $\mu =-\infty .$%
\quad $\square $

\begin{remark}
\label{rinser}For $\mu ,\nu \in \mathbb{R}_{\max }$, we have $\nu <+\infty $%
, so~(\ref{resid-pinf}) shows that 
\begin{equation*}
\mu \backslash \nu =+\infty \Leftrightarrow \mu =-\infty .
\end{equation*}
Hence, for $x,y\in \mathbb{R}_{\max }^{n}$, we have the following
equivalence 
\begin{equation}
x\backslash y=+\infty \Leftrightarrow x=-\infty \quad (\text{that is, }%
x_{i}=-\infty ,\;\forall i\in \lbrack n]).  \label{resid3}
\end{equation}
\end{remark}

Since $\mu \backslash \nu $ is an element of $\overline{\mathbb{R}}$, we get
by taking the complementaries of the equivalences (\ref{resid-finite}) and (%
\ref{resid-pinf}):

\begin{corollary}
\label{cone}We have 
\begin{eqnarray*}
\mu \backslash \nu =-\infty \Leftrightarrow (\mu \text{ or }\nu \notin 
\mathbb{R})\text{ and }\mu >-\infty ,\text{ }\nu <+\infty  \\
\Leftrightarrow (\mu =+\infty \text{ and }\nu <+\infty )\text{ or }(\mu
>-\infty \text{ and }\nu =-\infty )\enspace.
\end{eqnarray*}
\end{corollary}

By the above, we can summarize all possible values of $\mu \backslash \nu $
in the following table: 
\begin{equation*}
\begin{tabular}{|l|lll|}
\hline
\backslashbox{$\mu$}{$\nu$} & $-\infty $ & real & $+\infty $ \\ \hline
$-\infty $ & $+\infty $ & $+\infty $ & $+\infty $ \\ 
real & $-\infty $ & real & $+\infty $ \\ 
$+\infty $ & $-\infty $ & $-\infty $ & $+\infty $ \\ \hline
\end{tabular}
\end{equation*}

\begin{remark}
\label{rmurd}\emph{a) }Definition~(\ref{def-resid}) gives that for any
vectors $x=(x_{1},\ldots ,x_{n})^{T},$\newline
$y=(y_{1},\ldots ,y_{n})^{T}\in \overline{\mathbb{R}}_{\max }^{n}$, we have 
\begin{eqnarray}
{x}\backslash {y}&=&\max \{\lambda \in \overline{\mathbb{R}}_{\max }\mid
x\lambda \leq y\}  \label{resid-vec} \\
&=&\max \{\lambda \in \overline{\mathbb{R}}_{\max }\mid x_{i}\otimes \lambda
\leq y_{i}\;(i\in \lbrack n])\}  \notag \\
&=&\wedge _{i\in \lbrack n]}x_{i}\backslash y_{i}\enspace,  \notag
\end{eqnarray}
where $[n]=\{1,\ldots ,n\}$ and $\wedge $ denotes the infimum operation.
Hence, using also (\ref{resid-scalar-usual}), 
\begin{eqnarray}
{x}\backslash {y}&=&\max \{\lambda \in \overline{\mathbb{R}}\mid x_{i}+\lambda
\leq y_{i}\;(i\in \lbrack n])\}  \notag \\
&=&\min_{i\in \lbrack n]}(y_{i}+^{\prime }(-x_{i})).  \label{resid-usual}
\end{eqnarray}

\emph{b) }By (\ref{resid-usual}) and (\ref{conven-min}), for any $x\in 
\overline{\mathbb{R}}_{\max }^{n}$ we have 
\begin{equation}
x\backslash x=\wedge _{i\in \lbrack n]}x_{i}\backslash x_{i}=\min_{i\in
\lbrack n]}(x_{i}+^{\prime }(-x_{i}))=\left\{ 
\begin{array}{l}
+\infty \quad \text{if }x\in \{-\infty ,+\infty \}^{n} \\ 
0\quad \text{if }x\notin \{-\infty ,+\infty \}^{n}.
\end{array}
\right.   \label{resid-xx}
\end{equation}

\emph{c)} By (\ref{resid-vec}) we have the following equivalence: 
\begin{equation}
\lambda \leq x\backslash y\Leftrightarrow x\lambda \leq y,  \label{resid2}
\end{equation}
for all $\lambda \in \overline{\mathbb{R}}_{\max }$ and $x,y\in \overline{%
\mathbb{R}}_{\max }^{n}$ and for all $\lambda \in \mathbb{R}_{\max }$ and $%
x,y\in \mathbb{R}_{\max }^{n}$.
\end{remark}

\subsection{Hilbert projective distance}

For a complete semimodule $X$ over a complete idempotent semiring $\mathbb{S}
$ and for any $x,y\in X$, let us set 
\begin{equation}
\delta (x,y):=(x\backslash y)\otimes (y\backslash x)\enspace,  \label{delta}
\end{equation}
where $\otimes$ denotes the multiplication of $\mathbb{S}$. The last part of
the following result of \cite{CGQ} shows that when $\mathbb{S}\ $is
commutative, the mapping $\delta :X\times X\rightarrow \mathbb{S}$ \
satisfies an inequality opposite to the triangular inequality for a
distance, and thus $\delta (x,y)$ may be called an ``anti-distance''; by
abuse of language, we shall also keep this term in the non-commutative case,
even when $\delta $ is not symmetrical. Recall that, since $\mathbb{S}$ is
complete, the partial order relation defined by~\eqref{e-def-order}
determines an infimum operation, denoted by $\wedge$, see~\cite{CGQ}. In
what follows, we denote by $\mathbf{1}$ the unit element of $\mathbb{S}$.

\begin{proposition}
\cite[Theorem 17]{CGQ} \label{prop-anti-dist} Let $X$ be a complete
semimodule over a complete idempotent commutative semiring $\mathbb{S}$.
Then, for any $x,y,z\in X$, we have 
\begin{eqnarray}
\delta (x,y)&\leq& (x\backslash x)\wedge (y\backslash y),  \label{delta-1} \\
\delta (x,y)=\mathbf{1}&\Rightarrow& y=x\lambda ,\;\text{for some }\lambda \in 
\mathbb{S},  \label{delta-2} \\
\delta (x,z)&\geq& \delta (x,y)\otimes \delta (y,z)\enspace.
\label{antitriang}
\end{eqnarray}
\end{proposition}

Following \cite{Ben-01,CGQ,GK}, we define the Hilbert projective distance $d$
on $\overline{\mathbb{R}}_{\max }^{n}$ by 
\begin{equation}
d(x,y):=\delta (x,y)^{-},  \label{projdist2}
\end{equation}
with $\delta $ of~(\ref{delta}), that is, by the same expression~(\ref
{projdist}) as on $\mathbb{R}_{\max }^{n}$, where the superscript $-$ is
defined on $\overline{\mathbb{R}}$ \ by~(\ref{oppo}). For brevity, in the
sequel by ``distance'' we shall always mean the Hilbert projective distance.

\begin{corollary}
\label{d-distance}For $x,y,z\in \overline{\mathbb{R}}_{\max }^{n}\setminus
\{-\infty ,+\infty \}^{n},$ we have 
\begin{eqnarray}
d(x,y)&\geq& 0,\quad \quad   \label{d22} \\
d(x,y)=0&\Leftrightarrow& x=y\lambda ,\text{ for some }\lambda \in \mathbb{R}%
\enspace,  \label{d23} \\
d(x,z)&\leq& d(x,y)+d(y,z)\enspace.  \label{d24}
\end{eqnarray}

More generally, for all $x,y,z\in \overline{\mathbb{R}}_{\max }^{n}$, we
have 
\begin{equation}
d(x,z)\leq d(x,y)+^{\prime }d(y,z)\enspace,  \label{trei}
\end{equation}
and the implication 
\begin{equation}
d(x,y)=0\Rightarrow x=y\lambda ,\text{for some }\lambda \in \mathbb{R}%
\enspace.  \label{doi}
\end{equation}
\end{corollary}

(Recall that in the present setting, $y\lambda$ is now the vector with
entries $y_i+\lambda$, for $i\in [n]$.)

This result means that if one avoids vectors with only infinite entries,
then $d$ satisfies all properties of a projective distance, except that it
may take infinite values. The term ``projective'' comes from (\ref{d23}).
This result was given without proof in \cite[p.\ 6]{Ben-01}. For the sake of
completeness, we give here a proof, using Proposition~\ref{prop-anti-dist}.

\begin{proof}
(\ref{d22}): If $x,y\in \overline{\mathbb{R}}_{\max }^{n}\setminus \{-\infty
,+\infty \}^{n}$, then by (\ref{resid-xx}) we have $x\backslash
x=y\backslash y=0$, and hence by (\ref{delta-1}) we obtain $\delta (x,y)\leq
0\wedge 0=0$, so $d(x,y)=\delta (x,y)^{-}\geq 0$.

(\ref{doi}), (\ref{d23}): If $x,y\in \overline{\mathbb{R}}_{\max }^{n}$ and $%
d(x,y)=0,$ then $\delta (x,y)=d(x,y)^{-}=0^{-}=0=\mathbf{1},$ and hence by (%
\ref{delta-2}) there exists $\lambda \in \overline{\mathbb{R}}_{\max }$ such
that $y=x\lambda $. Thus, we have the implication $\Rightarrow $ in (\ref
{doi}) and (\ref{d23}). Conversely, if $y=x\lambda ,$ where $x,y\in 
\overline{\mathbb{R}}_{\max }^{n}\setminus \{-\infty ,+\infty \}^{n}$ and $%
\lambda \in \mathbb{R}$ , then $x\backslash y=x\backslash (x\lambda
)=(x+\lambda )-x=\lambda $ and $y\backslash x=$ $(x\lambda )\backslash
x=x-(x+\lambda )=-\lambda =\lambda ^{-\text{ }}$, whence 
\begin{equation*}
d(x,y)=d(x,x\lambda )=((x\backslash (x\lambda ))\otimes ((x\lambda
)\backslash x))^{-}=\lambda ^{-}\otimes \lambda =0.
\end{equation*}

(\ref{trei}), (\ref{d24}): Taking the opposite of (\ref{antitriang}) and
using (\ref{sim}), we get 
\begin{eqnarray*}
d(x,z)&=&\delta (x,z)^{-}\leq (\delta (x,y)\otimes \delta (y,z))^{-} \\
&=&\delta (y,z)^{-}\otimes ^{\prime }\delta (x,y)^{-}=d(y,z)+^{\prime }d(x,y),
\end{eqnarray*}
that is, (\ref{trei}). Finally, if $x,y,z\in \overline{\mathbb{R}}_{\max
}^{n}\setminus \{-\infty ,+\infty \}^{n},$ then by (\ref{d22}) all the
quantities in (\ref{trei}) are nonnegative, hence (\ref{trei}) reduces to (%
\ref{d24}).
\end{proof}

\begin{remark}
\label{remark-distance} \emph{a)} (\ref{d22}) does not hold for $x=y\in
\{-\infty ,+\infty \}^{n}$. Indeed, then $x\backslash x=+\infty $ (by (\ref
{resid-xx})), whence $\delta (x,x)=+\infty $, so 
\begin{equation}
d(x,x)=\delta (x,x)^{-}=-\infty ,\quad \quad \forall x\in \{-\infty ,+\infty
\}^{n}\enspace.  \label{minusinfin}
\end{equation}
Consequently, if $V$ is a subsemimodule of $\overline{\mathbb{R}}_{\max }^{n}
$ (so $-\infty \in V),$ then by (\ref{minusinfin}) we have
\begin{equation}
d(-\infty ,V)=\inf_{v\in V}d(-\infty ,v)=-\infty ,  \label{ciudat}
\end{equation}
which shows that the best approximation of $x=-\infty $ by $V$ is trivial.

\emph{b)} The implication $\Leftarrow $ of~(\ref{d23}) does not hold if $%
x\in \{-\infty ,+\infty \}^{n}$, since then the right hand side of (\ref{d23}%
) implies that $x=y$, but then by \emph{a)} above, $d(x,y)=-\infty $, so
that the left hand side of~(\ref{d23}) does not hold.

\emph{c)} In general one cannot replace $+^{\prime }$ by $+$ in (\ref{trei}%
). Indeed, for example if $x=y\in \{-\infty ,+\infty \}^{n}$ and $z\in 
\mathbb{R}^{n},$ then $d(x,z)=+\infty $ and $d(x,y)=-\infty $ (by (\ref
{minusinfin})), $d(y,z)=+\infty ,$ so $d(x,z)\nleqslant $ $d(x,y)+d(y,z).$

\emph{d)} In the sequel in the proofs of some statements about $d$ we shall
rather work with $\delta $ instead of $d$ in order to use only $+$ instead
of both $+^{\prime }$ and $+,$ and then only in the final step of the proof
we shall pass to the conclusion for $d(x,y)=\delta (x,y)^{-}$.
\end{remark}

In~\cite{CGQ} the results are presented in the case of complete
subsemimodules of a complete semimodule. Here we shall consider mainly the
projection onto, and the best approximation by elements, of a \emph{b}%
-complete subsemimodule of $\mathbb{R}_{\max }^{n}$, but as far as possible,
we shall write the results also for complete subsemimodules of $\overline{%
\mathbb{R}}_{\max }^{n}$. Recall that a subsemimodule $V$ of $\overline{%
\mathbb{R}}_{\max }^{n}$ is called \emph{complete} if the supremum (or
infinite sum) of any subset of $V$ belongs to $V$ and the scalar
multiplication $\otimes =+$ distributes over all infinite sums. If $V$ is a
subset (and in particular a subsemimodule) of $\overline{\mathbb{R}}_{\max
}^{n}$, and $x\in \overline{\mathbb{R}}_{\max }^{n}$, one simply defines the 
\emph{distance} $d(x,V)$ of $x$ to $V$ and the \emph{best approximation}\/
of $x$ by an element $v_{0}$ of $V$ as in the case of $\mathbb{R}_{\max
}^{n} $, that is by~(\ref{D3}) and~(\ref{D4}). For a complete subsemimodule $%
V$ of $\overline{\mathbb{R}}_{\max }^{n}$, the \emph{canonical projection
operator}\/ $P_{V}$ of $\overline{\mathbb{R}}_{\max }^{n}$ onto $V$ is also
defined \cite{CGQ} by~(\ref{phx1}) for all $x\in \overline{\mathbb{R}}_{\max
}^{n}$.

\subsection{An equivalent reformulation of the inequality $Ax\geq Bx$}

For a later application of our main distance formula to the solution of the
system of inequalities 
\begin{equation}
Ax\geq Bx,  \label{vac}
\end{equation}
where $A,B:\mathbb{R}_{\max }^{n}\rightarrow \mathbb{R}_{\max }^{p}$ are $%
p\times n$ matrices with entries in $\mathbb{R}_{\max }$, we give here an
equivalent reformulation of (\ref{vac}). We denote by $A_{i}$ and $B_{i}$
the $i$th rows of $A$ and $B$, respectively. We recall that for $B$ as
above, which in addition satisfies the assumption 
\begin{equation}
\text{ for all }j\in \lbrack n]\text{ there exists }i\in \lbrack p]\text{
such that }B_{ij}\not=-\infty ,  \label{resi0}
\end{equation}
one defines (see e.g.\ \cite{BCOQ}, \cite{AGG} and the references therein)
the \emph{residuated operator} $B^{\#}$ from $\mathbb{R}_{max}^{p}$ to $%
\mathbb{R}_{max}^{n}$ by 
\begin{equation}
(B^{\#}y)_{j}=\inf_{i}(-B_{ij}+^{\prime }y_{i}).  \label{resi1}
\end{equation}

We shall assume that the matrix $B$ satisfies assumption (\ref{resi0}).

The term ``residuated'' refers to the well-known equivalence of the
inequalities 
\begin{equation}
Bx\leq y\Leftrightarrow x\leq B^{\#}y.  \label{resi2}
\end{equation}
We recall the easy proof of this equivalence: We have 
\begin{eqnarray*}
Bx &\leq &y\Leftrightarrow (Bx)_{i}\leq y_{i},\;\forall 1\leq i\leq
p\Leftrightarrow \sup_{j}(B_{ij}+x_{j})\leq y_{i},\;\forall 1\leq i\leq p \\
&\Leftrightarrow &B_{ij}+x_{j}\leq y_{i},\;\forall 1\leq i\leq p,\;\forall
1\leq j\leq n \\
&\Leftrightarrow &x_{j}\leq -B_{ij}+^{\prime }y_{i},\text{ }\forall 1\leq
i\leq p,\;\forall 1\leq j\leq n \\
&\Leftrightarrow &x_{j}\leq \inf_{i}(-B_{ij}+^{\prime
}y_{i})=(B^{\#}y)_{j},\;\forall 1\leq j\leq n\Leftrightarrow x\leq B^{\#}y.
\end{eqnarray*}

Note that (\ref{resi1}) is a particular case of residuation operators for
matrices in the sense (\ref{def-resid}), since regarding $y\in \mathbb{R}%
_{\max }^{p}$ as a $p\times 1$ matrix we have 
\begin{equation*}
B^{\#}y=B\backslash y;
\end{equation*}
indeed, using (\ref{resi2}), we obtain 
\begin{equation*}
B^{\#}y=\max \{x\in \mathbb{R}_{\max }^{p}|x\leq B^{\#}y\}=\max \{x\in 
\mathbb{R}_{\max }^{p}|Bx\leq y|=B\backslash y.
\end{equation*}

Applying (\ref{resi2}) to $y=Ax,$ we get 
\begin{equation}
Bx\leq Ax\Leftrightarrow x\leq B^{\#}Ax.  \label{resi3}
\end{equation}

Finally, since the right hand side of (\ref{resi3}) can be written in the
form of the equality $x=B^{\#}Ax\wedge x,$ we obtain the equivalence 
\begin{equation}
Bx\leq Ax\Leftrightarrow x=B^{\#}Ax\wedge x,  \label{resi4}
\end{equation}
which we shall use later on.

\section{On the distance to a subsemimodule of $\mathbb{R}_{\max }^{n}$}
\label{s03}

Next we shall give some properties of the distance $d$ of (\ref{projdist})
and we shall show that one may reduce the study of the best approximation of
elements $x\in \mathbb{R}_{\max }^{n}$ by the elements of a \emph{b}%
-complete subsemimodule $V$ of $\mathbb{R}_{\max }^{n}$ to the case where $%
x\in \mathbb{R}^{n}$ and $V\subset \mathbb{R}^{n}\cup \{-\infty \}$.

\begin{definition}
\label{dsupport}For an element $x=(x_{1},\dots ,x_{n})^{T}$ of $\overline{%
\mathbb{R}}_{\max }^{n}$, we define the \emph{support}\/ $\func{Supp}\,x$, 
\emph{lower support}\/ $\func{Lsupp}\,x$ and \emph{upper support}\/ $\func{%
Usupp}\,x$ of $x$ by: 
\begin{eqnarray*}
\func{Supp}\,x&:=&\{i\in \lbrack n]\mid x_{i}\in \mathbb{R}\}, \\
\func{Lsupp}\,x&:=&\{i\in \lbrack n]\mid x_{i}<+\infty \}, \\
\func{Usupp}\,x&:=&\{i\in \lbrack n]\mid x_{i}>-\infty \}.
\end{eqnarray*}
\end{definition}

We have trivially 
\begin{equation}
\func{Supp}\,x=\func{Lsupp}\,x\cap \func{Usupp}\,x.  \label{triv}
\end{equation}
Moreover, when $x\in \mathbb{R}_{\max }^{n},$ we have 
\begin{equation*}
\func{Lsupp}\,x=[n].
\end{equation*}

\begin{lemma}
\label{l10}For any $x,y\in \overline{\mathbb{R}}_{\max }^{n}$ the following
statements are equivalent:

\begin{enumerate}
\item  \label{l10-1} $x\backslash y>-\infty $.

\item  \label{l10-2} There exists $\lambda \in \mathbb{R}$ such that $%
x\lambda \leq y$, that is, $x_{i}+\lambda \leq y_{i}$ for all $i\in \lbrack
n]$.

\item  \label{l10-3} We have 
\begin{equation}
\func{Usupp}\,x\subset \func{Usupp}\,y,\quad \func{Lsupp}\,x\supset \func{%
Lsupp}\,y.  \label{inclus}
\end{equation}
\end{enumerate}
\end{lemma}

\begin{proof}
\ref{l10-1}$^{\circ }\Rightarrow $ \ref{l10-2}$^{\circ }$. Let $x,y\in 
\overline{\mathbb{R}}_{\max }^{n}$ be such that $x\backslash y>-\infty $.
Then there exists $\lambda \in \mathbb{R}$ such that $\lambda \leq
x\backslash y$ and hence by (\ref{resid2}), $x\lambda \leq y$.

\ref{l10-2}$^{\circ }\Rightarrow $ \ref{l10-3}$^{\circ }$. Let $x,y\in 
\overline{\mathbb{R}}_{\max }^{n}$ and $\lambda \in \mathbb{R}$ be such that 
$x\lambda \leq y$, that is, $x_{i}+\lambda \leq y_{i}$ for all $i\in \lbrack
n]$. It follows that if $x_{i}>-\infty $ then $y_{i}>-\infty $, which shows
the inclusion $\func{Usupp}\,x\subset \func{Usupp}\,y$. Similarly, if $%
y_{i}<+\infty $ then $x_{i}<+\infty $, which shows the inclusion $\func{Lsupp%
}y\subset \func{Lsupp}\,x$.

\ref{l10-3}$^{\circ }\Rightarrow $ \ref{l10-1}$^{\circ }$. Assume that $%
x,y\in \overline{\mathbb{R}}_{\max }^{n}$ satisfy (\ref{inclus}). Since $%
x\backslash y=\min_{i\in \lbrack n]}x_{i}\backslash y_{i}$, we get that $%
x\backslash y>-\infty $ if and only $x_{i}\backslash y_{i}>-\infty $ for all 
$i\in \lbrack n]$. Now, if $y_{i}=-\infty $ then $i\in \lbrack n]\setminus 
\func{Usupp}\,y$ so by (\ref{inclus}), $i\in \lbrack n]\setminus \func{Usupp}%
\,x$, that is $x_{i}=-\infty ,$ whence by (\ref{resid-pinf}), $%
x_{i}\backslash y_{i}=+\infty >-\infty $. Similarly, if $x_{i}=+\infty $
then $i\in \lbrack n]\setminus \func{Lsupp}\,x$ so by (\ref{inclus}), $i\in
\lbrack n]\setminus \func{Lsupp}\,y$, that is $y_{i}=+\infty ,$ whence by (%
\ref{resid-pinf}), $x_{i}\backslash y_{i}=+\infty >-\infty $. Otherwise, $%
y_{i}>-\infty $ and $x_{i}<+\infty $, so there exist $\lambda $ and $\mu \in 
\mathbb{R}$ such that $y_{i}\geq \lambda $ and $x_{i}\leq \mu $, whence $%
x_{i}\backslash y_{i}\geq \mu \backslash \lambda \in \mathbb{R}$.
\end{proof}

As for convex sets without lines in linear spaces (see e.g.\ \cite
{Bauer,Thompson}), we define the \emph{part}\/ of an element of $\overline{%
\mathbb{R}}_{\max }^{n}$ as follows:

\begin{definition}
\label{dpart}The \emph{part}\/ $[[x]]$ of $x=(x_{1},\dots ,x_{n})^{T}\in 
\overline{\mathbb{R}}_{\max }^{n}$ is the equivalence class of $x$ for the
equivalence relation (of comparability) 
\begin{align*}
x\sim y\text{ if there exist }\lambda ,\mu \in \mathbb{R}\text{ such that }%
x\lambda \leq y\leq x\mu , \\
\text{that is, }x_{i}+\lambda \leq y_{i}\leq x_{i}+\mu \;\forall i\in
\lbrack n].\text{ }
\end{align*}
\end{definition}

Applying Lemma \ref{l10} symmetrically on $x$ and $y$, we deduce

\begin{lemma}
\label{l1}The following statements are equivalent for $x,y\in \overline{%
\mathbb{R}}_{\max }^{n}$:

\begin{enumerate}
\item  \label{l1-1} $d(x,y)<+\infty $.

\item  \label{l1-2} $x,y$ are in the same part.

\item  \label{l1-3} We have 
\begin{equation}
\func{Usupp}\,x=\func{Usupp}\,y,\quad \func{Lsupp}\,x=\func{Lsupp}\,y.
\label{egalit}
\end{equation}

\item  \label{l1-4} We have 
\begin{equation*}
\func{Supp}\,x=\func{Supp}\,y,\;\sigma _{-\infty }(x)=\sigma _{-\infty }(y)%
\text{ and }\sigma _{+\infty }(x)=\sigma _{+\infty }(y)\enspace,
\end{equation*}
where for any $x\in \overline{\mathbb{R}}_{\max }^{n}$ and $\lambda =\pm
\infty $ we denote 
\begin{equation*}
\sigma _{\lambda }(x):=\{i\in \lbrack n]\mid x_{i}=\lambda \}.
\end{equation*}
\end{enumerate}
\end{lemma}

\begin{proof}
\ref{l1-1}$^{\circ }\Rightarrow $ \ref{l1-2}$^{\circ }$. If $d(x,y)<+\infty $%
, then $(x\backslash y)\otimes (y\backslash x)=\delta
(x,y)=d(x,y)^{-}>-\infty $, which implies that both $x\backslash y$ and $%
y\backslash x$ are $>-\infty $, since $-\infty \otimes \mu =\mu \otimes
-\infty =-\infty $ for all $\mu \in \overline{\mathbb{R}}_{\max }$. Hence by
the implication \ref{l10-1}$^{\circ }\Rightarrow $\ref{l10-2}$^{\circ }$ of
Lemma~\ref{l10}, we get that there exist $\lambda ,\mu \in \mathbb{R}$ such
that $x\lambda \leq y$ and $y\mu \leq x$. Since $\mu \in \mathbb{R}$ is
invertible, it follows that $x\lambda \leq y\leq x\mu ^{-1}$, hence $x$ and $%
y$ are in the same part.

\ref{l1-2}$^{\circ }\Rightarrow $ \ref{l1-3}$^{\circ }$. Assume \ref{l1-2}$%
^{\circ },$ so there exist $\lambda ,\mu \in \mathbb{R}$ such that $x\lambda
\leq y\leq x\mu $. By the implication \ref{l10-2}$^{\circ }\Rightarrow $\ref
{l10-3}$^{\circ }$ of Lemma~\ref{l10}, we get the inclusions (\ref{inclus}).
But, since $\mu ,\lambda \in \mathbb{R}$ are invertible, we have $y\mu
^{-1}\leq x\leq y\lambda ^{-1}$, so that \ref{l1-2}$^{\circ }$ also holds
for the pair $(y,x)$. Consequently, we get the opposite inclusions to (\ref
{inclus}), whence the equalities (\ref{egalit}).

\ref{l1-3}$^{\circ }\Rightarrow $ \ref{l1-1}$^{\circ }$. Assume now the
equalities (\ref{egalit}). By the implication \ref{l10-3}$^{\circ
}\Rightarrow $ \ref{l10-1}$^{\circ }$ of Lemma~\ref{l10} applied to the two
pairs $(x,y)$ and $(y,x)$, we get that $x\backslash y>-\infty $ and $%
y\backslash x>-\infty $. This implies that $\delta (x,y)=(x\backslash
y)\otimes (y\backslash x)>-\infty $, whence $d(x,y)=\delta (x,y)^{-}<+\infty 
$.

\ref{l1-3}$^{\circ }\Rightarrow $ \ref{l1-4}$^{\circ }$. This follows from
the fact that for all $x\in \overline{\mathbb{R}}_{\max }^{n}$ we have (\ref
{triv}) and 
\begin{equation*}
\sigma _{-\infty }(x)=\func{Lsupp}\,x\setminus \func{Usupp}\,x,\quad \sigma
_{+\infty }(x)=\func{Usupp}\,x\setminus \func{Lsupp}\,x.
\end{equation*}

\ref{l1-4}$^{\circ }\Rightarrow $ \ref{l1-3}$^{\circ }$. Similarly this
follows from the fact that for all $x\in \overline{\mathbb{R}}_{\max }^{n}$
we have 
\begin{equation*}
\func{Lsupp}\,x=\func{Supp}\,x\cup \sigma _{-\infty }(x),\quad \func{Usupp}%
\,x=\func{Supp}\,x\cup \sigma _{+\infty }(x).\qed
\end{equation*}
\renewcommand{\qed}{}
\end{proof}

\begin{remark}
\label{rutil}\emph{a)} The equivalence \ref{l1-2}$^{\circ }\Leftrightarrow $ 
\ref{l1-1}$^{\circ }$ of Lemma \ref{l1} can be expressed in the form of the
following useful formula for the part of $x$: 
\begin{equation}
\lbrack \lbrack x]]=\{y\in \overline{\mathbb{R}}_{\max }^{n}|d(x,y)<+\infty
\}\quad \quad \forall x\in \overline{\mathbb{R}}_{\max }^{n}\enspace.
\label{useful}
\end{equation}
Hence, in particular, for any subset $V$ of $\overline{\mathbb{R}}_{\max
}^{n}$ we have 
\begin{equation}
V\cap \lbrack \lbrack x]]=\{v\in V|d(x,v)<+\infty \}\quad \quad \forall x\in 
\overline{\mathbb{R}}_{\max }^{n}\enspace.  \label{useful3}
\end{equation}

b) Similarly, the equivalences \ref{l1-2}$^{\circ }\Leftrightarrow $ \ref
{l1-3}$^{\circ }$ and \ref{l1-2}$^{\circ }\Leftrightarrow $ \ref{l1-4}$%
^{\circ }$of Lemma \ref{l1} can be expressed as formulas for the part of $x,$
namely: 
\begin{eqnarray*}
\lbrack \lbrack x]]&=&\{y\in \overline{\mathbb{R}}_{\max }^{n}|\func{Usupp}\,y=%
\func{Usupp}\,x,\;\func{Lsupp}\,y=\func{Lsupp}\,x\},\\
\lbrack \lbrack x]]&=&\{y\in \overline{\mathbb{R}}_{\max }^{n}|\func{Supp}\,y=%
\func{Supp}\,x,\;\sigma _{\lambda }(y)=\sigma _{\lambda }(x)\;(\lambda =\pm
\infty )\;\}.
\end{eqnarray*}
\end{remark}

\begin{corollary}
For each $x\in \overline{\mathbb{R}}_{\max }^{n}\setminus \{-\infty ,+\infty
\}^{n}$ we have $[[x]]\subset \overline{\mathbb{R}}_{\max }^{n}\setminus
\{-\infty ,+\infty \}^{n}$ and $d$ is a projective distance on $[[x]]$.
\end{corollary}

\begin{proof}
Let $x\in \overline{\mathbb{R}}_{\max }^{n}\setminus \{-\infty ,+\infty
\}^{n}$ and $y\in \lbrack \lbrack x]].$ If $y\in \{-\infty ,+\infty \}^{n}$,
then $\func{Supp}\,y=\{i\in \lbrack n]\mid y_{i}\in \mathbb{R}\}=\emptyset ,$
whence by $y\in \lbrack \lbrack x]]$ and the implication \ref{l1-2}$^{\circ
}\Rightarrow $ \ref{l1-4}$^{\circ }$ of Lemma \ref{l1}, we obtain $\func{Supp%
}x=\emptyset ,$ so $x\in \{-\infty ,+\infty \}^{n},$ which contradicts our
assumption. Therefore we must have $y\notin \{-\infty ,+\infty \}^{n},$
which proves the first assertion of the corollary. Finally, the second
assertion of the corollary holds by (\ref{useful}).
\end{proof}

\begin{corollary}
\label{part-infinite} For any $x\in \{-\infty ,+\infty \}^{n}$ the part of $x
$ is reduced to the singleton $\{x\}$, that is: 
\begin{equation}
\lbrack \lbrack x]]=\{x\}\quad \quad \forall x\in \{-\infty ,+\infty \}^{n}%
\enspace,  \label{plum}
\end{equation}
and hence, in particular, $\{-\infty \}$ is a part of $\overline{\mathbb{R}}%
_{\max }^{n}$ or $\mathbb{R}_{\max }^{n}$. Also, on $[[x]]$, $d$ is
identically equal to $-\infty $.
\end{corollary}

\begin{proof}
For any $x\in \{-\infty ,+\infty \}^{n}$ we have $\func{Supp}\,x=\emptyset $
and all the entries of $x$ are determined by $\sigma _{-\infty }(x)$: 
\begin{equation*}
x_{i}=\left\{ 
\begin{array}{l}
-\infty \quad \text{for }i\in \sigma _{-\infty }(x) \\ 
+\infty \quad \text{for }i\notin \sigma _{-\infty }(x).
\end{array}
\right. 
\end{equation*}
Then, the equivalence \ref{l1-2}$^{\circ }\Leftrightarrow $ \ref{l1-4}$%
^{\circ }$ of Lemma~\ref{l1} implies that $y\in \lbrack \lbrack x]]$ if and
only if $y=x$, which shows that $[[x]]=\{x\}$. Hence in particular, $%
[[-\infty ]]=\{-\infty \}$, so $\{-\infty \}$ is a part of $\overline{%
\mathbb{R}}_{\max }^{n}$ or $\mathbb{R}_{\max }^{n}$. Also, by (\ref
{minusinfin}), $d$ is identically equal to $-\infty $ on $[[x]]$, for any $%
x\in \{-\infty ,+\infty \}^{n}.$
\end{proof}

The main application to best approximation is the following:

\begin{theorem}
\label{reduc} If $V$ is a subset of $\mathbb{R}_{\max }^{n}$ (or $\overline{%
\mathbb{R}}_{\max }^{n}$) and $x\in \mathbb{R}_{\max }^{n}$ (or $\overline{%
\mathbb{R}}_{\max }^{n}$), then $d(x,V)<+\infty $ if and only if $V$
intersects the part $[[x]]$ of $x$ (i.e., $V\cap \lbrack \lbrack x]]\neq
\emptyset $), and in that case 
\begin{equation}
d(x,V)<d(x,v)=+\infty ,\quad \forall v\in V\setminus \lbrack \lbrack x]],
\label{f1}
\end{equation}
so any best approximation of $x$ in $V$ is necessarily in $[[x]]$, and 
\begin{equation}
d(x,V)=d(x,V\cap \lbrack \lbrack x]]).  \label{rebagat}
\end{equation}
\end{theorem}

\begin{proof}
Assume that $d(x,V)<+\infty $. Then $\inf_{v\in V}d(x,v)<+\infty $, so there
exists $v\in V$ such that $d(x,v)<+\infty $. By (\ref{useful}), we must have 
$v\in \lbrack \lbrack x]]$, so $V\cap \lbrack \lbrack x]]\not=\emptyset $.

Conversely, assume that $V\cap \lbrack \lbrack x]]\not=\emptyset $, say $%
v\in V\cap \lbrack \lbrack x]]$. Then by $v\in \lbrack \lbrack x]]$ and (\ref
{useful}), we have $d(x,v)<+\infty $, so by $v\in V$ we obtain $d(x,V)\leq
d(x,v)<+\infty $. This proves the equivalence $d(x,V)<+\infty
\Leftrightarrow V\cap \lbrack \lbrack x]]\neq \emptyset $. Moreover, by (\ref
{useful}) we have $d(x,v)=+\infty $ when $v\notin \lbrack \lbrack x]]$,
which shows formula (\ref{f1}), whence also $d(x,V\setminus \lbrack \lbrack
x]])=\inf_{v\in V\setminus \lbrack \lbrack x]]}d(x,v)=+\infty $, and any
best approximation of $x$ in $V$ is necessarily in $[[x]]$. Since $V$ is the
disjoint union $V=(V\cap \lbrack \lbrack x]])\cup (V\setminus \lbrack
\lbrack x]]),$ we obtain 
\begin{equation*}
d(x,V)=\min \{d(x,V\cap \lbrack \lbrack x]]),d(x,V\backslash \lbrack \lbrack
x]])\}=d(x,V\cap \lbrack \lbrack x]]).
\end{equation*}
\end{proof}

The first part of Theorem \ref{reduc} can be also expressed in the following
useful form:

\begin{corollary}
\label{cutil2}For any subsemimodule $V$ of $\overline{\mathbb{R}}_{\max }^{n}
$ we have 
\begin{equation*}
\{x\in \overline{\mathbb{R}}_{\max }^{n}|d(x,V)<+\infty \}=\{x\in \overline{%
\mathbb{R}}_{\max }^{n}|V\cap \lbrack \lbrack x]]\neq \emptyset \}.
\end{equation*}
\end{corollary}

In the sequel we shall give some results in $\mathbb{R}_{\max }^{n}.$

\begin{corollary}[in $\mathbb{R}_{\max }^{n}$]
\label{rmaimu2} \emph{a)} For $x\in \mathbb{R}_{\max }^{n}$, we have $%
[[x]]\subset \mathbb{R}_{\max }^{n}$ and 
\begin{equation}
\lbrack \lbrack x]]=\{y\in \mathbb{R}_{\max }^{n}\mid {\func{Supp}\,}y={%
\func{Supp}\,}x\}\enspace.  \label{part2}
\end{equation}
\emph{b)} For $x\in \mathbb{R}_{\max }^{n}$, we have 
\begin{equation}
d(x,-\infty )=
\begin{cases}
+\infty  & \text{if }x>-\infty  \\ 
-\infty  & \text{if }x=-\infty .
\end{cases}
\label{kix}
\end{equation}
\end{corollary}

\begin{proof}
\emph{a)} By the implication \ref{l1-2}$^{\circ }\Rightarrow $ \ref{l1-3}$%
^{\circ }$ of Lemma \ref{l1} and the obvious equivalence 
\begin{equation}
x\in \mathbb{R}_{\max }^{n}\Leftrightarrow (x\in \overline{\mathbb{R}}_{\max
}^{n},\;\func{Lsupp}\,x=[n]),  \label{obvi}
\end{equation}

we get for any $x\in \mathbb{R}_{\max }^{n}$ that $y\in \lbrack \lbrack x]]$
implies $\func{Lsupp}\,y=\func{Lsupp}\,x=[n],$ whence $y\in \mathbb{R}_{\max
}^{n};$ thus $[[x]]\subset \mathbb{R}_{\max }^{n}$ for any element $x\in $ $%
\mathbb{R}_{\max }^{n}$. Moreover, if $x\in \mathbb{R}_{\max }^{n}$, then $%
\func{Usupp}\,x={\func{Supp}\,}x$ and hence, by (\ref{obvi}) and the
equivalence \ref{l1-2}$^{\circ }\Leftrightarrow $\ref{l1-3}$^{\circ }$ of
Lemma \ref{l1} we have $y\in \lbrack \lbrack x]]$ if and only if ${\func{Supp%
}\,}y={\func{Supp}}x.$

\emph{b)} If $x>-\infty $ then $\func{Supp}\,x\neq \emptyset ,$ and since $%
\func{Supp}\,(-\infty )=\emptyset $, $x$ and $-\infty $ are in different
parts. Hence, by (\ref{useful}), $d(x,-\infty )=+\infty $. On the other
hand, by (\ref{minusinfin}) we have $d(-\infty ,-\infty )=-\infty $.
\end{proof}

\begin{remark}
\label{distance-rn} \emph{a)} In particular, $\mathbb{R}^{n}$ \emph{is a
part of }$\mathbb{R}_{\max }^{n}$; indeed, the points in $\mathbb{R}^{n}$
are exactly those that have support equal to the set $[n],$ and hence, by (%
\ref{part2}), all points in $\mathbb{R}^{n}$ are in the same part as one of
them, say $0.$

\emph{b) }When $x=(x_{1},\dots ,x_{n})^{T},y=(y_{1},\dots ,y_{n})^{T}\in 
\mathbb{R}_{\max }^{n}\setminus \{-\infty \}$ have the same support $%
I\subset \lbrack n]$, we have 
\begin{equation}
d(x,y)=\max_{i\in I}(x_{i}-y_{i})-\min_{j\in I}(x_{j}-y_{j}).  \label{D1p}
\end{equation}

Indeed, if $k\notin I:=\func{Supp}\,x=\func{Supp}\,y$, then $%
x_{k}=y_{k}=-\infty ,$ so $x_{k}\backslash y_{k}=+\infty $ (see~(\ref{resid3}%
)), and hence 
\begin{eqnarray*}
x\backslash y &=&\min_{i\in \lbrack n]}(x_{i}\backslash y_{i})=\min
\{\min_{i\in I}(x_{i}\backslash y_{i}),\min_{k\in \lbrack n]\backslash
I}(x_{k}\backslash y_{k})\} \\
&=&\min_{i\in I}(x_{i}\backslash y_{i})=\min_{i\in I}(y_{i}-x_{i}),
\end{eqnarray*}
where the terms in the latter expression are all finite. Similarly, $%
y\backslash x=\min_{j\in I}(x_{j}-y_{j}).$ Consequently, we obtain 
\begin{eqnarray*}
d(x,y) &=&((x\backslash y)\otimes (y\backslash x))^{-} \\
&=&-(\min_{i\in I}(y_{i}-x_{i})+\min_{j\in I}(x_{j}-y_{j})) \\
&=&\max_{i\in I}(x_{i}-y_{i})-\min_{j\in I}(x_{j}-y_{j}),
\end{eqnarray*}
that is, (\ref{D1p}).

\emph{c) }Combining Remark \ref{rutil} and Corollary \ref{rmaimu2}a), it
follows that in $\mathbb{R}_{\max }^{n}$we have the equivalence 
\begin{equation}
d(x,y)<+\infty \Leftrightarrow {\func{Supp}\,}y={\func{Supp}\,}x.
\label{cemasa}
\end{equation}
Note that this also follows from Lemma \ref{l1}, equivalence \ref{l1-1}$%
^{\circ }\Leftrightarrow $\ref{l1-4}$^{\circ }$.
\end{remark}

\begin{remark}
When $x,y\in \mathbb{R}_{\max }^{n}\setminus \{-\infty \}$, the Hilbert
projective distance $d(x,y)$ can be characterized by 
\begin{equation}
d(x,y)=\inf \{\frac{\mu }{\lambda }\mid \lambda \in \mathbb{R},\mu \in 
\mathbb{R}\qquad y\lambda \leq x\leq y\mu \}\enspace.
\label{e-d-hilbert}
\end{equation}
To show this, we shall assume that $x$ and $y$ have the same support
(otherwise, the set in~\eqref{e-d-hilbert} is empty, so its infimum, $%
+\infty $, trivially coincides with $d(x,y)=+\infty $). Then, the maximal $%
\lambda \in \mathbb{R}$ such that $y\lambda \leq x$ is $\min_{j\in
I}(x_{j}-y_{j})$. Similarly, the term $\max_{i\in I}(x_{i}-y_{i})$ coincides
with the minimal $\mu \in \mathbb{R}$ such that $x\leq y\mu $. Therefore, %
\eqref{e-d-hilbert} coincides with the expression of $d(x,y)$ in~\eqref{D1p}.
\end{remark}

Using formula (\ref{phx21}), we get as a corollary of Theorem~\ref{reduc}:

\begin{corollary}
\label{pvinclassx}\emph{a)} If $V$ is a \emph{b}-complete subsemimodule of $%
\mathbb{R}_{\max }^{n}$, and $d(x,V)<+\infty $, then $P_{V}(x)\in V\cap
\lbrack \lbrack x]]$.

\emph{b) }Consequently, if $V$ is a \emph{b}-complete subsemimodule of $%
\mathbb{R}_{\max }^{n}$, and $V\cap \lbrack \lbrack x]]\neq \emptyset ,$
then $P_{V}(x)\in V\cap \lbrack \lbrack x]]$.
\end{corollary}

\begin{proof}
a) We have $P_{V}(x)\in V$ by the definition (\ref{phx1}) of $P_{V}(x)$.
Furthermore, by (\ref{phx21}) and our assumption we have $%
d(x,P_{V}(x))=d(x,V)<+\infty ,$ and hence by (\ref{f1}) we obtain $%
P_{V}(x)\in \lbrack \lbrack x]]$.

b) This follows from (\ref{useful3}) and part a).
\end{proof}

\begin{proposition}
\label{pabra}If $V$ is a (\emph{b}-complete) subsemimodule of $\mathbb{R}%
_{\max }^{n}$, then the set 
\begin{equation}
V^{(x)}:=(V\cap \lbrack \lbrack x]])\cup \{-\infty \}  \label{f22}
\end{equation}
is the smallest (\emph{b}-complete) subsemimodule of $\mathbb{R}_{\max }^{n}$
containing $V\cap \lbrack \lbrack x]]$ and we have 
\begin{equation}
d(x,V)=d(x,V^{(x)}).  \label{f3}
\end{equation}
\end{proposition}

\begin{proof}
Assume that $V$ is a \emph{b}-complete subsemimodule of $\mathbb{R}_{\max
}^{n}$ and $x\in \mathbb{R}_{\max }^{n}.$ Since any subsemimodule of $%
\mathbb{R}_{\max }^{n}$ contains necessarily $-\infty $, any subsemimodule
of $\mathbb{R}_{\max }^{n}$ containing $V\cap \lbrack \lbrack x]]$
necessarily contains $V^{(x)}=(V\cap \lbrack \lbrack x]])\cup \{-\infty \}$.
Moreover, since $-\infty \in V$, by (\ref{f22}) we have $(V\cap \lbrack
\lbrack x]])\subset V^{(x)}\subset V$, whence, by (\ref{rebagat}) and (\ref
{kix}), we obtain~(\ref{f3}).

Now let us prove that $V^{(x)}$ is a \emph{b}-complete subsemimodule of $%
\mathbb{R}_{\max }^{n}$. It is easy to see that $[[x]]\cup \{-\infty \}$ is
a subsemimodule of $\mathbb{R}_{\max }^{n}$, since for all $y,z\in \mathbb{R}%
_{\max }^{n}$ and $\lambda \in \mathbb{R}_{\max }$, we have $\func{Supp}%
\,(y\oplus z)=\func{Supp}\,y\cup \func{Supp}\,z$ and $\func{Supp}\,y\lambda =%
\func{Supp}\,y$ if $\lambda \neq -\infty $ and $\func{Supp}\,y\lambda
=\emptyset $ otherwise. Hence, since $V$ is a subsemimodule of $\mathbb{R}%
_{\max }^{n}$, so is also the set 
\begin{equation}
V^{(x)}=(V\cap \lbrack \lbrack x]])\cup (V\cap \{-\infty \})=V\cap
([[x]]\cup \{-\infty \}).  \label{vex}
\end{equation}
Let us show that $V^{(x)}$ is \emph{b}-complete. Let $M$ be a subset of $%
V^{(x)}$ bounded from above by an element of $V^{(x)}$. Since $%
V^{(x)}\subset V$, then $M$ is also a subset of $V$ bounded from above by an
element of $V$, and since $V$ is \emph{b}-complete, then $M$ admits a
supremum in $V$. Let us denote it by $m$ and show that it belongs to $V^{(x)}
$. If $m=-\infty $, then $m\in V^{(x)}$ and we are done. Otherwise, there
exists $y\in M\setminus \{-\infty \}\subset \lbrack \lbrack x]]$. Since $%
m\geq y$, we get that $\func{Supp}\,m\supset \func{Supp}\,y$ (by the
implication \ref{l10-2}$^{\circ }\Rightarrow $\ref{l10-3}$^{\circ }$ of
Lemma~\ref{l10}) and since $\func{Supp}\,y=\func{Supp}\,x$ for all $y\in
\lbrack \lbrack x]]$ (by Corollary~\ref{rmaimu2}, a)), we obtain $\func{Supp}%
\,m\supset \func{Supp}\,x$. Conversely, if $i\not\in \func{Supp}\,x$, then $%
y_{i}=-\infty $ for all $y\in \lbrack \lbrack x]]\cup \{-\infty \}$, hence
for all $y\in M\subset V^{(x)}\subset \lbrack \lbrack x]]\cup \{-\infty \}$,
which implies that $m_{i}=\sup \{y_{i}\mid y\in M\}=-\infty $. This shows
that $\func{Supp}\,m\subset \func{Supp}\,x$, hence the equality, which is
equivalent to the property that $m\in \lbrack \lbrack x]]$ (again by
Corollary~\ref{rmaimu2}, a)). This implies that $m\in \lbrack \lbrack
x]]\cap V\subset V^{(x)}$, and shows that $V^{(x)}$ is a \emph{b}-complete
subsemimodule of $\mathbb{R}_{\max }^{n}$.
\end{proof}

Now we shall show that one can reduce the study of the best approximation of
elements $x\in \mathbb{R}_{\max }^{n}$ by the elements of a \emph{b}%
-complete subsemimodule $V$ of $\mathbb{R}_{\max }^{n}$ to the case where 
\begin{equation}
x\in \mathbb{R}^{n^{\prime }},V\subset \mathbb{R}^{n^{\prime }}\cup
\{-\infty \},  \label{case}
\end{equation}
with a suitable $n^{\prime }\leq n$ depending on $x.$ To this end, for any $%
I\subset \lbrack n]$ and $x\in \mathbb{R}_{\max }^{n}$, let us denote by $%
x|_{I}$ the image of $x$ by the restriction $r_{I}$ to coordinates in $I$: 
\begin{equation}
r_{I}:\mathbb{R}_{\max }^{n}\rightarrow \mathbb{R}_{\max }^{I},\;x\mapsto
x|_{I}:=(x_{i})_{i\in I}\enspace.  \label{restrict}
\end{equation}
We shall also use the notation 
\begin{equation}
V|_{I}:=\{v|_{I}\mid v\in V\}.  \label{restrict-set}
\end{equation}

\begin{lemma}
\label{restr-supp} Let $I\subset \lbrack n]$ and denote 
\begin{equation*}
M_{I}:=\{y\in \mathbb{R}_{\max }^{n}\mid \func{Supp}\,y\subset I\}\enspace.
\end{equation*}
\emph{a) }$r_{I}$ is injective on $M_{I}.$

\emph{b) }For all $y,z\in M_{I}$, we have 
\begin{equation*}
d(y,z)=\widetilde{d}(y|_{I},z|_{I}),
\end{equation*}
where in the right hand side $\widetilde{d}$ is the Hilbert projective
distance on $\mathbb{R}_{\max }^{I}$.

\emph{c)} If $y\in \mathbb{R}_{\max }^{n},W\subset M_{I}$, then $d(y,W)=%
\widetilde{d}(y|_{I},W|_{I})$.
\end{lemma}

\begin{proof}
a) Let $y^{\prime },y^{\prime \prime }\in M_{I}$ be such that $%
r_{I}(y^{\prime })=r_{I}(y^{\prime \prime }),$ so
$\func{Supp}\,y^{\prime },%
\func{Supp}\,y^{\prime \prime }\subset I,$ $y_{i}^{\prime }=y_{i}^{\prime
\prime }\;(i\in I).$ Then $y_{j}^{\prime }=y_{j}^{\prime \prime }=-\infty $
for all $j\notin I,$ and hence $y^{\prime }=y^{\prime \prime }.$ Thus $r_{I}$
is injective on $M_{I}.$

b) The second assertion follows from the fact that for $y,z\in M_{I},\lambda
\in \mathbb{R}_{\max }$, $y\lambda \leq z$ if and only if $y|_{I}\lambda
\leq z|_{I}.$ Indeed, we have 
\begin{eqnarray*}
d(y,z)&=&((y\backslash z)\otimes (z\backslash y))^{-} \\
&=&(\sup \{\lambda \in \mathbb{R}_{\max }|y\lambda \leq z\}\otimes \sup \{\mu
\in \mathbb{R}_{\max }|z\mu \leq y\})^{-} \\
&=&(\sup \{\lambda \in \mathbb{R}_{\max }|y|_{I}\lambda \leq z|_{I}\}\otimes
\sup \{\mu \in \mathbb{R}_{\max }|z|_{I}\mu \leq y|_{I}\})^{-} \\
&=&((y|_{I}\backslash z|_{I})\otimes (z|_{I}\backslash y|_{I}))^{-}=\widetilde{%
d}(y|_{I},z|_{I}).
\end{eqnarray*}

c) For the last assertion, let $w\in W$. Since then $w\in M_{I}$, from b) we
get that $d(y,w)=\widetilde{d}(y|_{I},w|_{I})$, whence, since $%
W|_{I}=r_{I}(W),$ we obtain 
\begin{equation*}
d(y,W)=\inf_{w\in W}d(y,w)=\inf_{w\in W}\widetilde{d}(y|_{I},w|_{I})=%
\inf_{w^{\prime }\in W|_{I}}\widetilde{d}(y|_{I},w^{\prime })=\widetilde{d}%
(y|_{I},W|_{I}).
\end{equation*}
\end{proof}

\begin{proposition}
\label{reduc-cor} Let $V$ be a \emph{b}-complete subsemimodule of $\mathbb{R}%
_{\max }^{n}$ and $x\in \mathbb{R}_{\max }^{n}$ such that $d(x,V)<+\infty $.
Define 
\begin{equation}
x^{\prime }:=x|_{\func{Supp}\,x}\in \mathbb{R}^{\func{Supp}\,x},V^{\prime
}:=V^{(x)}|_{\func{Supp}\,x}\subset \mathbb{R}^{\func{Supp}\,x}\cup
\{-\infty |_{\func{Supp}\,x}\},  \label{feti}
\end{equation}
where $V^{(x)}:=(V\cap \lbrack \lbrack x]])\cup \{-\infty \}$ (of (\ref{f22}%
)). Then $V^{\prime }$ is a \emph{b}-complete subsemimodule of $\mathbb{R}%
_{\max }^{\func{Supp}\,x}=\mathbb{R}^{\func{Supp}\,x}\cup \{-\infty |_{\func{%
Supp}\,x}\},x^{\prime }\in \mathbb{R}_{\max }^{\func{Supp}\,x},$ and we have 
\begin{equation}
d(x,V)=\widetilde{d}(x^{\prime },V^{\prime })\enspace,  \label{yehu}
\end{equation}
where in the right hand side $\widetilde{d}$ is the Hilbert projective
distance on $\mathbb{R}_{\max }^{\func{Supp}\,x}.$ Furthermore, an element $%
v\in V$ is a best approximation of $x$ in $V$ if and only if $\func{Supp}\,v=%
\func{Supp}\,x$ and $v^{\prime }:=v|_{\func{Supp}\,x}$ is a best
approximation of $x^{\prime }$ in $V^{\prime }$.
\end{proposition}

\begin{proof}
Clearly $x^{\prime }=x|_{\func{Supp}\,x}\in \mathbb{R}^{\func{Supp}\,x}$ and
by Corollary \ref{rmaimu2}a) we have Supp $v=$Supp $x$ for all $v\in \lbrack
\lbrack x]],$ whence 
\begin{equation*}
V^{\prime }=[(V\cap \lbrack \lbrack x]])\cup \{-\infty \}]|_{\func{Supp}%
\,x}\subset \mathbb{R}^{\func{Supp}\,x}\cup \{-\infty |_{\func{Supp}\,x}\}.
\end{equation*}

Furthermore, since $V$ is a \emph{b}-complete subsemimodule of $\mathbb{R}%
_{\max }^{n},$ so is $V^{(x)}$ of (\ref{f22}) (by Proposition \ref{pabra})
and hence, since $V^{\prime }=r_{\func{Supp}\,x}(V^{(x)}),$ where $r_{\func{%
Supp}\,x}:\mathbb{R}_{\max }^{n}\rightarrow \mathbb{R}_{\max }^{\func{Supp}%
\,x}$ is a max-linear mapping, $V^{\prime }$ is a \emph{b}-complete
subsemimodule of $\mathbb{R}_{\max }^{\func{Supp}\,x}$.

By Proposition~\ref{pabra}, we have $d(x,V)=d(x,V^{(x)})$ and the supports
of the elements of $V^{(x)}$ are all included in $\func{Supp}\,x$, so by
Lemma~\ref{restr-supp}, we get that $d(x,V)=\widetilde{d}(x^{\prime
},V^{\prime })$, where $\widetilde{d}$ is the Hilbert projective distance on 
$\mathbb{R}_{\max }^{\func{Supp}\,x}$.

Assume now that $v\in V$ is a best approximation of $x$ by $V,$ that is, $%
d(x,v)=d(x,V).$ Then by Theorem \ref{reduc}, $v\in \lbrack \lbrack x]],$
whence by Corollary \ref{rmaimu2}a), $\func{Supp}\,v=\func{Supp}\,x;$ also, $%
v\in V\cap \lbrack \lbrack x]]\subset V^{(x)},$ whence $v|_{\func{Supp}%
\,x}\in V^{(x)}|_{\func{Supp}\,x}=V^{\prime }.$ Therefore, using Lemma~\ref
{restr-supp}, we obtain 
\begin{equation*}
\widetilde{d}(x^{\prime },v|_{\func{Supp}\,x})=d(x,v)=d(x,V)=\widetilde{d}%
(x^{\prime },V^{\prime }),
\end{equation*}
so $v^{\prime }:=v|_{\func{Supp}\,x}$ is a best approximation of $x^{\prime }
$ in $V^{\prime }.$

Conversely, assume now that $v\in \mathbb{R}_{\max }^{n}$ is such that $%
\func{Supp}\,v=\func{Supp}\,x$ and $v|_{\func{Supp}\,x}$ is a best
approximation of $x^{\prime }$ in $V^{\prime }$, that is, $\widetilde{d}%
(x^{\prime },v|_{\func{Supp}\,x})=\widetilde{d}(x^{\prime },V^{\prime })$.
Then $v|_{\func{Supp}\,x}=r_{\func{Supp}\,x}(v)$ (by the injectivity of $r_{%
\func{Supp}\,x}$, see Lemma~\ref{restr-supp}), whence $v\in V^{(x)}\subset V,
$ and using Lemma~\ref{restr-supp} we obtain 
\begin{equation*}
d(x,v)=\widetilde{d}(x^{\prime },v|_{\func{Supp}\,x})=\widetilde{d}%
(x^{\prime },V^{\prime })=d(x,V),
\end{equation*}
so $v$ is a best approximation of $x$ in $V$.
\end{proof}

\begin{remark}
\label{rchiar}Denoting by $n^{\prime }$ the cardinality of $\func{Supp}\,x$
and using the isomorphism between $\mathbb{R}^{n^{\prime }}$ and $\mathbb{R}%
^{\func{Supp}\,x},$ Proposition \ref{reduc-cor} shows that one can reduce
the study of the best approximation of elements $x\in \mathbb{R}_{\max }^{n}$
by the elements of a \emph{b}-complete subsemimodule $V$ of $\mathbb{R}%
_{\max }^{n}$ to the case (\ref{case}). Practically, given $V$ and $x\notin
V,$ whence also $d(x,V),$ if we want to find a best approximation of $x$ by $%
V,$ one can pass to $x^{\prime }=x|_{\func{Supp}\,x}$ and $V^{\prime
}=V^{(x)}|_{\func{Supp}\,x},$ then find a best approximation $v^{\prime }$
of $x^{\prime }$ in $V^{\prime },$ and then, by the above, the element $%
v=(v_{1},\dots ,v_{n})\in V$ defined by 
\begin{equation*}
v_{i}=\left\{ 
\begin{array}{l}
v_{i}^{\prime }\quad \text{if }i\in \func{Supp}\,x \\ 
-\infty \quad \text{if }i\notin \func{Supp}\,x
\end{array}
\right. 
\end{equation*}
will be a best approximation of $x$ by $V.$
\end{remark}

\section{Further results on the universal separation theorem and
applications to best approximation}\label{s04}

In classical linear analysis, one first reduces the problem of best
approximation of elements $x$ by linear subspaces $V$ to the case of
suitable half-spaces $H=H_{V,x}$ that separate $V$ and $x$. %
In this section we shall apply a similar method to best approximation of $%
x\in \mathbb{R}_{\max }^{n}$ by elements of subsemimodules $V$ of $\mathbb{R}%
_{\max }^{n}.$ The relevant notion of half-space used for separation depends
on the framework in which we are working. When considering best
approximation by complete subsemimodules $V$ of $\overline{\mathbb{R}}_{\max
}^{n},$ it is natural to use separation by \emph{complete half-spaces} of $%
\overline{\mathbb{R}}_{\max }^{n}$, while for best approximation by \emph{b}%
-complete subsemimodules $V$ of $\mathbb{R}_{\max }^{n}$ it is natural to
use separation by \emph{closed half-spaces} of $\mathbb{R}_{\max }^{n},$ as
we shall see below.

In~\cite{CGQ}, \cite{CGQS}, \cite{GK} and \cite{GS}, the separation theorems
for $\mathbb{R}_{\max }^{n}$ have been obtained as consequences of the
results of~\cite{CGQ} concerning complete semimodules. We shall follow here
a similar approach, deducing the separation and best approximation results
in $\mathbb{R}_{\max }^{n}$ from separation and best approximation results
in complete semimodules, since the proofs are more transparent in the latter
setting.

\begin{theorem}
\cite[Theorem 8]{CGQ} \label{th-univ} Let $X$ be a complete semimodule over
the complete idempotent semiring $\mathbb{S}$. Let $V$ be a complete
subsemimodule of $X$, $x\in X$ and $x\notin V$, and consider the set 
\begin{equation}
K:=\{h\in X\mid h\backslash x\leq h\backslash P_{V}(x)\}\enspace.
\label{bau}
\end{equation}
Then $V\subset K$ and $x\not\in K$.
\end{theorem}

\begin{remark}
\label{rscrisa}In~\cite{CGQ}, the result is written with the equality 
\begin{equation}
h\backslash x=h\backslash {P_{V}(x)}  \label{eorth}
\end{equation}
in~(\ref{bau}); however, by a remark made in \cite{GS} for \emph{b}-complete
semimodules, which is valid also for complete semimodules,since $%
P_{V}(x)\leq x$, the inequality $h\backslash P_{V}(x)\leq h\backslash x$
holds for all $h\in X$, hence the two formulations are equivalent.
\end{remark}

In~\cite{GS}, a \emph{half-space} of a complete semimodule $X$ is defined as
a set of the form 
\begin{equation}
K=K_{u,v}:=\{h\in X\mid h\backslash u\leq h\backslash v\}\enspace,
\label{comphalfsp2}
\end{equation}
with $u,v\in X$. Note that all half-spaces $K_{u,v}$ are complete
subsemimodules of $X$. To be correct with the terminology ``half-space'',
one should avoid the case where $K=X$, which holds if and only if $u\leq v$,
and the case where $K=\{\bot \}$ where $\bot $ is the smallest element of $X$
(which is also its neutral element for the addition $\oplus $). With this
definition, the set $K$ of Theorem~\ref{th-univ} is a (complete) half-space,
and when $x\notin V$, $K$ separates $x$ from $V$. We shall call it \emph{the
universal complete half-space of} $X$ \emph{separating} $x$ from $V$.

\begin{remark}
\label{rbasic}In particular, if $\mathbb{S}$ $=\overline{\mathbb{R}}_{\max },
$ the complete max-plus semiring, and $X=\overline{\mathbb{R}}_{\max }^{n}$,
(complete) half-spaces can be put in a more usual form, namely every
complete half-space $K=K_{u,v}$ as in (\ref{comphalfsp2}), with $u,v\in 
\overline{\mathbb{R}}_{\max }^{n},$ can be written in the form 
\begin{equation}
H_{a,b}:=\{h\in \overline{\mathbb{R}}_{\max }^{n}|ah\geq bh\},
\label{comphalfsp}
\end{equation}
with $a=(a_{1},\ldots ,a_{n}),b=(b_{1},\dots ,b_{n})\in \overline{\mathbb{R}}%
_{\max }^{1\times n},$ where the notation $ah$ means 
\begin{equation}
ah=\max_{i\in \lbrack n]}(a_{i}+h_{i})\enspace,  \label{scalprod}
\end{equation}
and conversely, every set $H=H_{a,b}$ as in (\ref{comphalfsp}) can be
written in the form (\ref{comphalfsp2}), with $u,v\in \overline{\mathbb{R}}%
_{\max }^{n}.$ Indeed, by taking $a=-u^{T},b=-v^{T},$ respectively $%
u=-b^{T},v=-a^{T},$ and using (\ref{resid-usual}), we have 
\begin{eqnarray*}
K_{u,v}&=&\{h\in \overline{\mathbb{R}}_{\max }^{n}|\min_{i\in \lbrack
n]}(u_{i}+^{\prime }(-h_{i}))\leq \min_{i\in \lbrack n]}(v_{i}+^{\prime
}(-h_{i}))\} \\
&=&\{h\in \overline{\mathbb{R}}_{\max }^{n}|\max_{i\in \lbrack
n]}(-u_{i}+h_{i})\geq \max_{i\in \lbrack n]}(-v_{i}+h_{i})\} \\
&=&\{h\in \overline{\mathbb{R}}_{\max }^{n}|(-u^{T})h\geq
(-v^{T})h\}=H_{-u^{T},-v^{T}}.
\end{eqnarray*}
\end{remark}

In~\cite{GK}, the universal separation theorem is written with half-spaces
of the form $H_{a,b}$. Later we shall call $a$ and $b$ the ``coefficients''
of the representation (\ref{comphalfsp}) of $H$.

Since 
\begin{equation}
\varphi _{a}(h):=ah\quad \quad \forall h\in \overline{\mathbb{R}}_{\max
}^{n},  \label{genform}
\end{equation}
where $a\in \mathbb{R}_{\max },$ is the general form of the max-linear forms
on $\overline{\mathbb{R}}_{\max }^{n}$ (e.g.\ by \cite{CGQ}, Theorem 36; see
also \cite{LMS}), we can also write 
\begin{equation*}
H=\{h\in \overline{\mathbb{R}}_{\max }^{n}|\varphi _{a}(h)\geq \varphi
_{b}(h)\}.
\end{equation*}

\begin{remark}
\label{half-space-forms}Let us mention that identifying $(\overline{\mathbb{R%
}}_{\max }^{n})^{\ast }$ with $\overline{\mathbb{R}}_{\max }^{n}$ in the
usual way, we may also regard $ah$ of~(\ref{scalprod}) as the ``max-plus
scalar product'' of two row vectors or of two column vectors; however, we
shall not use here this identification.
\end{remark}

Our next aim will be to show that for an element $x$ of a complete
semimodule (respectively of a \emph{b}-complete semimodule) $X,$ the
computation of the canonical projection onto, and the distance to, any
(respectively any \emph{b}-complete) subsemimodule $V$\ of $X$, can be
reduced to the computation of the canonical projection onto, and the
distance to, a complete half-space $K$ (respectively a closed half-space $%
H^{\prime })$ of $X.$

For a subset $M$ of any complete semimodule $X$ over a complete idempotent
semiring $\mathbb{S}$ and any $x\in M$ let us set 
\begin{equation}
\delta (x,M):=\sup_{v\in M}\delta (x,v);  \label{farthest0}
\end{equation}
then we may regard any $v_{0}\in M$ satisfying 
\begin{equation}
\delta (x,v_{0})\geq \delta (x,v),\quad \quad \forall v\in M,
\label{farthest}
\end{equation}
(or, equivalently, $\delta (x,v_{0})=\delta (x,M))$ as a ``farthest point''
in $M$ from $x,$ in the ``anti-distance'' $\delta .$

\begin{remark}
\label{drema}Since the Hilbert projective distance $d$ on $X$ is defined by (%
\ref{projdist2}), the relation (\ref{farthest}) is equivalent to $%
d(x,v_{0})^{-}\geq d(x,v)^{-}\;(v\in M),$ that is, to $d(x,v_{0})\leq
d(x,v)\quad (v\in M),$ meaning that $v_{0}$ is a nearest point in $M$ to $x$
in the Hilbert's projective metric $d.$ This remark will permit us to deduce
results on nearest points in Hilbert's projective metric $d$ from results on
farthest points in the anti-distance $\delta .$
\end{remark}

\begin{theorem}
\cite[Theorem 18]{CGQ} \label{th0b} If $V$ is a complete subsemimodule of a
complete semimodule $X$ over a complete idempotent semiring, and $x\in X,$
then 
\begin{equation*}
\delta (x,P_{V}(x))\geq \delta (x,v),\qquad \forall v\in V\enspace,
\end{equation*}
i.e., $P_{V}(x)$ is a farthest point from $x$ among the elements of $V$ in
the anti-distance $\delta $. \qed
\end{theorem}

We recall that the Hilbert's projective distance $d(x,V)$ between an element 
$x$ and a set $V$ is defined by (\ref{D3}).

\begin{corollary}
\label{cbaubum} If $V$ is a complete subsemimodule of a complete semimodule $%
X$ over a complete idempotent semiring, and if $x\in X,$ then we have~%
\textrm{(\ref{phx21})}, or, in other words, 
\begin{equation*}
d(x,P_{V}(x))\leq d(x,v),\qquad \forall v\in V\enspace,
\end{equation*}
i.e., $P_{V}(x)$ is a best approximation of $x$ in $V$ for Hilbert's
projective distance in $X.$ \qed
\end{corollary}

\begin{remark}
\label{cpartic} For $X=\overline{\mathbb{R}}_{\max }^{n}$, Corollary~\ref
{cbaubum} has been given in~\cite{GK}, Theorem 1.
\end{remark}

We next establish some additional properties of the universal separating
complete half-space and apply them to reduce the problem of best
approximation by subsemimodules to best approximation by half-spaces.

\begin{theorem}
\label{th1} If $V$ is a complete subsemimodule of a complete semimodule $X$
over a complete idempotent semiring $\mathbb{S}$, if $x\notin V$, and if $K$
is the associated complete half-space separating $x$ and $V$ (see Theorem~%
\ref{th-univ}), then 
\begin{equation*}
P_{V}(x)=P_{K}(x).
\end{equation*}
\end{theorem}

\begin{proof}
Since $K\supset V$, we have $P_{K}(x)\geq P_{V}(x)$. If $h\in K$ is such
that $h\leq x$, we have $\mathbf{1}\leq h\backslash x=h\backslash {P_{V}(x)}$%
, where $\mathbf{1}$ is the neutral element of $\otimes $ in $\mathbb{S},$
and so, $h\leq P_{V}(x)$. Since this holds for all $h\in K$ such that $h\leq
x$, it follows that $P_{K}(x)\leq P_{V}(x)$.
\end{proof}

\begin{corollary}
\label{label2} If $V$ is a complete subsemimodule of a complete semimodule $X
$ over a complete idempotent semiring, if $x\in X,x\notin V,$ and if $K$ is
the associated complete half-space \textrm{(\ref{bau}) }separating $x$ and $V
$, then we have 
\begin{equation*}
d(x,V)=d(x,K)\enspace.
\end{equation*}
\end{corollary}

\begin{proof}
Combining Corollary \ref{cbaubum} and Theorem \ref{th1}, we obtain 
\begin{equation*}
d(x,V)=d(x,P_{V}(x))=d(x,P_{K}(x))=d(x,K).\qed
\end{equation*}
\renewcommand{\qed}{}
\end{proof}

Finally, let us show the connection between the canonical projection and
orthogonality. The relation~(\ref{eorth}) can be thought of as an analogue
of the classical orthogonality relation $\langle h,x-P_{V}(x)\rangle =0,$
where $\langle .,.\rangle $ denotes the usual inner product, characterizing
the nearest point $P_{V}(x)$ of an element $x$ onto a linear subspace. We
next show that in the setting of semimodules, the canonical projection $%
P_{V}(x)$ is still characterized by the previous ``orthogonality'' property.

\begin{definition}
If $X$ is a complete idempotent semimodule, for $x,y,z\in X$ we shall say
that the ``bivector'' $(x,y)\in X^{2}$ is \emph{orthogonal to\ }$z,$\ and we
shall write $(x,y)\bot z,$\ if 
\begin{equation}
z\backslash x=z\backslash y.  \label{orthog}
\end{equation}
The bivector $(x,y)\in X^{2}$\ is said to be \emph{orthogonal to a subset}\ $%
M$\ of\ $X,$\ and we write $(x,y)\bot M,$ if $(x,y)\bot z\;$for all $z\in M.$
\end{definition}

In particular, if $X=\overline{\mathbb{R}}_{\max }^{n}$ and $%
x=(x_{1},\dots,x_{n})^{T},y=(y_{1},\dots,y_{n})^{T},z=(z_{1},%
\dots,z_{n})^{T}\in \overline{\mathbb{R}}_{\max }^{n},$ then by (\ref
{resid-usual}), the relation (\ref{orthog}) is equivalent to 
\begin{equation*}
\wedge _{i\in \lbrack n]}(x_{i}+^{\prime}(-z_{i}))=\wedge _{i\in \lbrack
n]}(y_{i}+^{\prime}(-z_{i})).
\end{equation*}
Theorem \ref{th-univ} shows that for any complete subsemimodule $V$ of $%
\overline{\mathbb{R}}_{\max }^{n}$ and any $x\notin V,$ the bivector $%
(x,P_{V}(x))$ is orthogonal to $V.$ Now we shall show that $P_{V}(x)$ is the
only element of $V$ with this property.

\begin{theorem}
Let $V$ be a complete subsemimodule of a complete semimodule $X$ over a
complete idempotent semiring, and let $x\in X,x\notin V.$ Then, $P_{V}(x)$
is the unique element $y$ of $V$ such that $(x,y)\bot V,$ i.e., such that 
\begin{equation}
v\backslash x=v\backslash y,\qquad \forall v\in V\enspace.  \label{eorth1}
\end{equation}
\end{theorem}

\begin{proof}
By Theorem~\ref{th-univ}, $y=P_{V}(x)$ satisfies the above relations. We
next show that $y$ is unique.

If $y$ satisfies~(\ref{eorth1}), then for all $v\in V$ we have $y\geq
v(v\backslash y)=v(v\backslash x)$, and so, $y\geq P_{V}(x)=\sup_{v\in
V}v(v\backslash x)$.

Moreover, taking $v=y$ in~(\ref{eorth1}), we get $y\backslash x=y\backslash
y\geq \mathbf{1}$, where $\mathbf{1}$ is the neutral element of $\mathbb{S}$
for $\otimes ,$ and so $x\geq y$. Since $P_{V}(x)$ is the maximal element of 
$V$ which is bounded above by $x$, it follows that $y\leq P_{V}(x)$. Hence $%
y=P_{V}(x)$.
\end{proof}

Let us pass now to $\mathbb{R}_{\max }^{n}.$ As mentioned at the beginning
of this section, when considering the \emph{b}-complete (but not complete)
semimodule $\mathbb{R}_{\max }^{n}$, instead of $\overline{\mathbb{R}}_{\max
}^{n}$, one is rather interested to take the \emph{closed half-spaces} of $%
\mathbb{R}_{\max }^{n}$ as tools for separation, which are defined as the
sets of the form 
\begin{eqnarray}
H^{\prime }=H_{a,b}^{\prime }&=&\{h\in \mathbb{R}_{\max }^{n}\mid ah\geq bh\}
\notag \\
&=&\{h\in \mathbb{R}_{\max }^{n}\mid \max_{i\in \lbrack n]}(a_{i}+h_{i})\geq
\max_{i\in \lbrack n]}(b_{i}+h_{i})\}\enspace,  \label{clohalfsp}
\end{eqnarray}
where $a=(a_{1},\dots,a_{n}),b=(b_{1},\dots,b_{n})\in (\mathbb{R}_{\max
}^{n})^{\ast }$ are row vectors with coordinates in $\mathbb{R}_{\max }.$ We
will call $H^{\prime }$ \emph{the universal closed half-space of }$\mathbb{R}%
_{\max }^{n}$ \emph{separating }$x$\emph{\ from} $V$. The term ``closed''
refers to the usual topology of $\mathbb{R}_{\max }^{n}$, since the set $%
H^{\prime }$ of (\ref{clohalfsp}) with $a,b\in (\mathbb{R}_{\max
}^{n})^{\ast }$ is always closed in $\mathbb{R}_{\max }^{n}$ (by \cite{CGQS}%
, Proposition 3.7). A particular case which will be important in the sequel
is that when $H_{a,b}^{\prime }$ \emph{has finite apex} (we recall that the
number $-(a\oplus b)$ is called \cite{joswig04} the \emph{apex} of $%
H_{a,b}^{\prime })$.

Any closed half-space $H_{a,b}^{\prime }$ of $\mathbb{R}_{\max }^{n}$ is the
trace over $\mathbb{R}_{\max }^{n}$ of a complete half-space of $\overline{%
\mathbb{R}}_{\max }^{n}$ (but not vice versa). Indeed, taking $X=\overline{%
\mathbb{R}}_{\max }^{n}$ thought of as a complete $\overline{\mathbb{R}}%
_{\max }$-semimodule, and taking $u=-a^{T}$ and $v=-b^{T}$, where $a,b\in (%
\mathbb{R}_{\max }^{n})^{\ast },$ by Remark~\ref{rbasic} we obtain 
\begin{equation*}
K_{u,v}\cap \mathbb{R}_{\max }^{n}=H_{a,b}\cap \mathbb{R}_{\max
}^{n}=H_{a,b}^{\prime }\enspace.
\end{equation*}

For \emph{b}-complete subsemimodules of $\mathbb{R}_{\max }^{n}$ we obtain
the following results:

\begin{corollary}
\label{coromindist} If $V$ is a \emph{b}-complete subsemimodule of $\mathbb{R%
}_{\max }^{n}$, and if $x\in \mathbb{R}_{\max }^{n}$, then 
\begin{equation*}
d(x,P_{V}(x))\leq d(x,v),\qquad \forall v\in V\enspace,
\end{equation*}
i.e., $P_{V}(x)$ is a best approximation of $x$ in $V$ for Hilbert's
projective distance in $\mathbb{R}_{\max }^{n}.$
\end{corollary}

\begin{proof}
Since $V$ is a \emph{b}-complete subsemimodule of $\mathbb{R}_{\max }^{n}$,
it has a completion $\hat{V}$ in $\overline{\mathbb{R}}_{\max }^{n}$, which
consists of the suprema of arbitrary subsets of $V$. The latter is a
complete subsemimodule of $\overline{\mathbb{R}}_{\max }^{n}$. It is readily
seen that $P_{\hat{V}}(x)=P_{V}(x)$, so the result follows from Corollary 
\ref{cbaubum}.
\end{proof}

\begin{corollary}
\label{coromindist2} If $V$ is a \emph{b}-complete subsemimodule of $\mathbb{%
R}_{\max }^{n}$, if $x\in \mathbb{R}_{\max }^{n},x\not\in V$, and if 
\begin{eqnarray}
H^{\prime }&=&\{h\in \mathbb{R}_{\max }^{n}\mid h\backslash x\leq h\backslash {%
P_{V}(x)}\}  \label{cefru} \\
&=&\{h\in \mathbb{R}_{\max }^{n}|\max_{j\in \lbrack n]}(h_{j}-x_{j})\geq
\max_{j\in \lbrack n]}(h_{j}-P_{V}(x)_{j})\},  \notag
\end{eqnarray}
then 
\begin{equation*}
P_{V}(x)=P_{H^{\prime }}(x).
\end{equation*}
\end{corollary}

Here, and in the sequel, for $a,b\in \overline{\mathbb{R}}_{\max}$, we set 
\begin{equation*}
a-b:=a+(-b) \enspace .
\end{equation*}

\begin{proof}
This follows similarly to Corollary \ref{coromindist}, using now Theorem \ref
{th1}.\ 
\end{proof}

\begin{corollary}
\label{chop}If $V$ is a \emph{b}-complete subsemimodule of $\mathbb{R}_{\max
}^{n}$, and if $x\in \mathbb{R}_{\max }^{n},x\not\in V$, then 
\begin{equation*}
d(x,V)=d(x,H^{\prime })\enspace,
\end{equation*}
with $H^{\prime }$ of (\ref{cefru}).
\end{corollary}

\begin{proof}
Applying Corollaries~\ref{coromindist} and~\ref{coromindist2}, we get 
\begin{equation*}
d(x,V)=d(x,P_{V}(x))=d(x,P_{H^{\prime }}(x))=d(x,H^{\prime }).\qed
\end{equation*}
\renewcommand{\qed}{}
\end{proof}

\section{The canonical projection onto, and the distance to, a closed
half-space of $\mathbb{R}_{\max }^{n}$}\label{s05}

In the next result we shall give an explicit formula for the canonical
projection onto a closed half-space of $\mathbb{R}_{\max }^{n}.$ To this
end, the following notation will be useful: If $b\in (\mathbb{R}_{\max
}^{n})^{\ast }$ is a row vector and $\lambda \in \mathbb{R}_{\max }$ is a
scalar, we set 
\begin{equation}
b\backslash {\lambda :}=\sup \{u\in \mathbb{R}_{\max }^{n}\mid bu\leq
\lambda \}\in \overline{\mathbb{R}}_{\max }^{n},  \label{bepele}
\end{equation}
$u$ being thought of as a column vector. So $b\backslash {\lambda }$ is a
column vector with entries 
\begin{eqnarray}
(b\backslash {\lambda )}_{j} &\!\!\!=\!\!\!&(\sup \{u\in \mathbb{R}_{\max }^{n}\mid
bu\leq \lambda \})_{j}  \notag \\
&\!\!\!=\!\!\!&\sup \{u\in \mathbb{R}_{\max }^{n}\mid b_{j}u\leq \lambda
\}=b_{j}\backslash {\lambda =}\left\{ 
\begin{array}{l}
(b_{j})^{-1}\lambda \quad \text{if }j\in \func{Supp}\,b \\ 
+\infty \text{\quad if }j\notin \func{Supp}\,b.
\end{array}
\right. \quad \quad  \label{bepel}
\end{eqnarray}
\textbf{\ }

\begin{theorem}
\label{prop1} Let $a,b\in (\mathbb{R}_{\max }^{n})^{\ast }$ be row vectors
and consider the closed half-space 
\begin{equation}
H=\{h\in \mathbb{R}_{\max }^{n}\mid ah\geq bh\}\enspace.  \label{clohalfsp2}
\end{equation}
Let $I=\func{Supp}\,a$, $J=\func{Supp}\,b,$ and assume $I\cap J=\emptyset $
and that $J\neq \emptyset $ $(b\neq -\infty )$. Then for any $x\in \mathbb{R}%
_{\max }^{n}$ we have 
\begin{equation}
P_{H}(x)=x\wedge ({b}\backslash {ax})\enspace,  \label{e0}
\end{equation}
i.e., 
\begin{eqnarray}
\!\!\!\!\!\!\!\!\!\!\!\!(P_{H}(x))_{j}&\!\!\!=\!\!\!&x_{j}\wedge ({b_{j}}\backslash {ax})  \notag \\
&\!\!\!=\!\!\!&
\begin{cases}
x_{j} & \text{for }j\in J^{c}, \\ 
x_{j}\wedge \Big(b_{j}^{-1}(ax)\Big)=x_{j}\wedge \left( b_{j}^{-1}\Big(%
\bigoplus_{i\in I}a_{i}x_{i}\Big)\right)  & \text{for }j\in J\enspace,
\end{cases}
\label{e1}
\end{eqnarray}
where $J^{c}$ denotes the complement of $J$ in $[n]$.
\end{theorem}

\begin{proof}
We set 
\begin{equation*}
u:=x\wedge (b\backslash ax)
\end{equation*}
and first observe that the coordinates of $u$ coincide with the right hand
side of~(\ref{e1}).

Assume that $h\in H$ is such that $x\geq h$. Then, $ax\geq ah\geq bh$, and
so, $h\leq \sup \{u^{\prime }\in \mathbb{R}_{\max }^{n}\mid bu^{\prime }\leq
ax\}=b\backslash {ax}$. It follows that $h\leq x\wedge (b\backslash ax)=u$.
This implies that $P_{H}(x)=\sup \{h\in H\mathbf{|}h\leq x\}\leq u$. To show
that the equality holds, it remains to check that $au\geq bu$. We have 
\begin{equation*}
bu=b(x\wedge (b\backslash ax))\leq b(b\backslash {ax})=b\sup \{u\in \mathbb{R%
}_{\max }^{n}|bu\leq ax\}\leq ax=\bigoplus_{i\in I}a_{i}x_{i}.
\end{equation*}
But by $I\cap J=\emptyset $ we have $I\subseteq J^{c},$ whence by (\ref{e1}%
), $a_{i}x_{i}=a_{i}x_{i}(b_{i}\backslash ax)=a_{i}u_{i}\;(i\in I),$ and
therefore $\bigoplus_{i\in I}a_{i}x_{i}=\bigoplus_{i\in I}a_{i}u_{i}=au.$
Thus, finally, $bu\leq au.$
\end{proof}

The following result gives the main formula for the distance to a closed
half-space:

\begin{theorem}
\label{th1new}Let $a,b\in (\mathbb{R}_{\max }^{n})^{\ast }$ be row vectors, $%
H$ the closed half-space (\ref{clohalfsp2}), and $x\not\in H$. Then 
\begin{equation}
d(x,H)={ax}\backslash {bx}=
\begin{cases}
(ax)^{-1}bx & \text{ if }ax\neq -\infty \enspace, \\ 
+\infty  & \text{ if }ax=-\infty .
\end{cases}
\label{dhalfsp}
\end{equation}
\end{theorem}

\begin{proof}
Since $P_{H}(x)$ maximizes the opposite of Hilbert's distance to $x$ among
the points of $H$, see~Theorem \ref{th0b}, we have 
\begin{equation}
\delta (x,H)=\delta (x,P_{H}(x))=({P_{H}(x)}\backslash {x})({x}\backslash {%
P_{H}(x)})\enspace.  \label{maxi}
\end{equation}

Assume first that $P_{H}(x)\neq -\infty $. We claim that in that case 
we have ${P_{H}(x)}\backslash {x}=0$. Indeed, since $P_{H}(x)\leq x$, we
must have ${\lambda }:={P_{H}(x)}\backslash {x}\geq 0$. Assume by
contradiction that ${\lambda }>0$. Then since $P_{H}(x)\neq -\infty $, $%
P_{H}(x){\lambda }>P_{H}(x)0=P_{H}(x)$. Since $H$ is a max-plus linear
subspace, and since $P_{H}(x)\in H$, we have $P_{H}(x){\lambda }\in H$, but
since by the definition of ${\lambda }$, $P_{H}(x){\lambda }\leq x$, this
contradicts the definition of $P_{H}(x)$ as the maximal element $h\in H$
such that $h\leq x$.

Then, using successively Equations~(\ref{maxi}),~(\ref{e0}), and residuation
properties of max-plus linear maps (see~\cite{CGQ}), we get 
\begin{eqnarray}
\delta (x,H)&=&({P_{H}(x)}\backslash {x})({x}\backslash {P_{H}(x)})={x}%
\backslash {P_{H}(x)}  \notag \\
&=&{x}\backslash {(x\wedge ({b}\backslash {ax}))}=({x}\backslash {x})\wedge ({x%
}\backslash {({b}\backslash {ax})})  \notag \\
&=&({x}\backslash {x})\wedge ({bx}\backslash {ax})\enspace.  \label{expr}
\end{eqnarray}
Since $x\not\in H$, we have $x\neq -\infty $ (because $-\infty \in \{h\in 
\mathbb{R}_{\max }^{n}\mid ah\geq bh\}=H),$ so ${x}\backslash {x}=0$. Also,
again since $x\notin H,$ we have $bx>ax$. Hence $bx\neq -\infty $ and ${bx}%
\backslash {ax}<0$, so (\ref{expr}) simplifies to 
\begin{equation*}
\delta (x,H)=0\wedge \left( {bx}\backslash {ax}\right) ={bx}\backslash {ax}%
=(bx)^{-1}ax\enspace.
\end{equation*}
Consequently, by (\ref{projdist2}), we arrive at 
\begin{equation*}
d(x,H)=(\delta (x,H))^{-1}={ax}\backslash {bx}\enspace.
\end{equation*}
Assume now that $P_{H}(x)=-\infty $. Then since $x\not\in H$, so $x\neq
-\infty $, we have, using (\ref{phx21}), that $d(x,H)=d(x,P_{H}(x))=d(x,-%
\infty )=+\infty $. Moreover, by~(\ref{e1}), we get that $%
x_{i}=P_{H}(x)_{i}=-\infty $ for all $i\not\in J$, so that in particular $%
a_{i}x_{i}=-\infty \;(i\notin J)$. Hence by the definition of $J$ and since $%
bx>ax$, we get that ${ax}\backslash {bx}=\sup \{\lambda \in \mathbb{R}_{\max
}|\lambda ax\leq bx\}=+\infty =d(x,H)$.
\end{proof}

\section{The canonical forms of closed half-spaces of $\mathbb{R}_{\max}^{n} $}
\label{s06}

We have the following result, which shows that every closed half-space (\ref
{clohalfsp}) of $\mathbb{R}_{\max }^{n}$ admits a canonical representation
with the aid of coefficients with disjoint supports:

\begin{proposition}
\label{prop0} Let $a,b\in (\mathbb{R}_{\max }^{n})^{\ast }\backslash
\{-\infty \}$ be row vectors such that $a\not\geq b$ and there exists $i\in
\lbrack n]$ such that $a_{i}\geq b_{i}$, and consider the closed half space 
\begin{equation}
H=\{h\in \mathbb{R}_{\max }^{n}|ah\geq bh\}\enspace  \label{totaia}
\end{equation}
(the assumptions on the coefficients $a$ and $b$ are equivalent to $%
\{-\infty \}\neq H\neq \mathbb{R}_{\max }^{n}$). Let $a^{\prime }$ and $%
b^{\prime }\in (\mathbb{R}_{\max }^{n})^{\ast }$ be the truncations of $a$
and $b$ defined by 
\begin{equation}
a_{i}^{\prime }=
\begin{cases}
a_{i} & \text{if }a_{i}\geq b_{i} \\ 
-\infty  & \text{if }a_{i}<b_{i},
\end{cases}
\;\quad b_{j}^{\prime }=
\begin{cases}
b_{j} & \text{if }a_{j}<b_{j} \\ 
-\infty  & \text{if }a_{j}\geq b_{j}.
\end{cases}
\label{trunc0}
\end{equation}
Then $\func{Supp}\,a^{\prime }\cap $ $\func{Supp}\,b^{\prime }=\emptyset ,$
and $H$ can be written in the form: 
\begin{equation}
H=\{h\in \mathbb{R}_{\max }^{n}\mid a^{\prime }h\geq b^{\prime }h\}\enspace.
\label{half-reduc}
\end{equation}
\end{proposition}

\begin{proof}
Let us denote 
\begin{equation*}
J:=\{j\in \lbrack n]\mid a_{j}<b_{j}\},\;J^{c}:=\{j\in \lbrack n]|a_{j}\geq
b_{j}\},
\end{equation*}
so that 
\begin{equation}
a_{i}^{\prime }=
\begin{cases}
a_{i} & \text{for }i\in J^{c} \\ 
-\infty  & \text{otherwise,}
\end{cases}
\;\quad b_{j}^{\prime }=
\begin{cases}
b_{j} & \text{for }j\in J \\ 
-\infty  & \text{otherwise.}
\end{cases}
\label{trunc}
\end{equation}
Thus, $\func{Supp}\,(a^{\prime })\subseteq J^{c}$ and $\func{Supp}%
\,(b^{\prime })\subseteq J,$ whence $\func{Supp}\,(a^{\prime })\cap \func{%
Supp}\,(b^{\prime })=\emptyset .$

Furthermore, let $H^{\prime }$ be the right hand side of~(\ref{half-reduc}),
and let us show that $H=H^{\prime }$. The elements $a^{\prime }$ and $%
b^{\prime }$ satisfy $a^{\prime }\leq a$ and $b^{\prime }\leq b$ and since $%
a_{i}^{\prime }=a_{i}\geq b_{i}$ for $i\in J^{c}$, and $b_{i}^{\prime }=b_{i}
$ for $i\in J$, we deduce that $b\leq a^{\prime }\oplus b^{\prime }$.

Let $h\in H^{\prime }$, then $b^{\prime }h\leq a^{\prime }h$. Hence $bh\leq
(a^{\prime }\oplus b^{\prime })h\leq a^{\prime }h\leq ah,$ so $h\in H,$
which shows the inclusion $H^{\prime }\subseteq H.$

Conversely, let $h\in H$, then $bh\leq ah$. Since $b^{\prime }\leq b$, this
implies that $b^{\prime }h\leq ah$. Let $a^{\prime \prime }$ be the
truncation of $a$ to $J$, then $a=a^{\prime }\oplus a^{\prime \prime }$ and
thus 
\begin{equation}
b^{\prime }h\leq a^{\prime }h\oplus a^{\prime \prime }h\enspace.
\label{bpapas}
\end{equation}
If the support of $h$ does not intersect $J$, then $a^{\prime \prime
}h=-\infty $ and~(\ref{bpapas}) implies that $b^{\prime }h\leq a^{\prime }h$%
, that is $h\in H^{\prime }$. Otherwise, since $a_{i}<b_{i}$ for all $i\in J$%
, we deduce that $a^{\prime \prime }h<b^{\prime }h$, hence by~(\ref{bpapas}%
), it follows that the maximum of $a^{\prime }h$ and $a^{\prime \prime }h$
which is greater or equal to $a^{\prime }h$, is necessarily equal to $%
a^{\prime }h$. Hence again $b^{\prime }h\leq a^{\prime }h$, and thus $h\in
H^{\prime }$. We have shown the converse inclusion $H\subseteq H^{\prime }$,
hence the equality.
\end{proof}

\begin{corollary}
\label{th1new2} Let $a,b$ and $H$ be as in Proposition~\ref{prop0} and
assume that $x\not\in H$. Then, 
\begin{equation*}
d(x,H)={a^{\prime }x}\backslash {bx}\enspace,
\end{equation*}
where $a^{\prime }$ is defined as in Proposition~\ref{prop0}.
\end{corollary}

\begin{proof}
Using Theorem~\ref{th1new} for the coefficients $a^{\prime },b^{\prime }$
defined in Proposition~\ref{prop0}, we get that $d(x,H)={a^{\prime }x}%
\backslash {b^{\prime }x}$. Since $x\not\in H$, we have $ax<bx$. Let $J$ and 
$b^{\prime }$ be defined as in Proposition~\ref{prop0}. Since $b_{i}\leq
a_{i}$ when $i\in J^{c}$, we deduce that $bx=b^{\prime }x$, which shows the
corollary.
\end{proof}

\begin{definition}
\label{dcan}We shall call (\ref{half-reduc}) the \emph{canonical form} of
the closed half-space $H.$
\end{definition}

For the computation of distances to, and elements of best approximation by,
subsemimodules, it is worthwhile to write explicitly the canonical form of
the universal separating closed half-space (\ref{cefru}) for a pair $(V,x),$
where $V$ is a \emph{b}-complete subsemimodule of $\mathbb{R}_{\max }^{n}$
and $x\notin V:$

\begin{corollary}
\label{cuniv}If $V$ is a \emph{b}-complete subsemimodule of $\mathbb{R}%
_{\max }^{n}$ and $x\in \mathbb{R}_{\max }^{n},x\notin V$ is such that all
coordinates of $P_{V}(x)$ (and hence also of $x$) are $>-\infty ,$ then the
following closed half-space separates $x$ from $V:$ 
\begin{eqnarray}
H_{V,x}^{\prime }&=&\{h\in \mathbb{R}_{\max }^{n}\mid
\max_{j|x_{j}=P_{V}(x)_{j}}(h_{j}-x_{j})\geq
\max_{j|x_{j}>P_{V}(x)_{j}}(h_{j}-(P_{V}(x)_{j}))\}  \notag \\
&=&\{h\in \mathbb{R}_{\max }^{n}|\wedge _{j\in J}h_{j}\backslash x_{j}\leq
\wedge _{j\in J^{c}}h_{j}\backslash P_{V}(x)_{j}\},  \label{kefe}
\end{eqnarray}
where 
\begin{equation}
J=\{j\in \lbrack n]\mid x_{j}=P_{V}(x)_{j}\},\;J^{c}=\{j\in \lbrack
n]|x_{j}>P_{V}(x)_{j}\}.  \label{au}
\end{equation}
\end{corollary}

\begin{proof}
This follows from Proposition \ref{prop0}, setting 
\begin{equation*}
a_{j}=-x_{j},\;b_{j}=-P_{V}(x)_{j}\quad \quad (j\in \lbrack n]).
\end{equation*}
Indeed, then $a_{j}\geq b_{j}\Leftrightarrow -x_{j}\geq
P_{V}(x)_{j}\Leftrightarrow x_{j}=P_{V}(x)_{j}$ (where the last equivalence
holds by $P_{V}(x)\leq x)$ and $a_{j}<b_{j}\Leftrightarrow
-x_{j}<-P_{V}(x)_{j}\Leftrightarrow x_{j}>P_{V}(x)_{j},$ whence by (\ref
{trunc0}), 
\begin{equation}
a_{j}^{\prime }=\left\{ 
\begin{array}{l}
-x_{j}\text{\quad if }x_{j}=P_{V}(x)_{j} \\ 
-\infty \quad \text{if }x_{j}>P_{V}(x)_{j},
\end{array}
\right. \quad b_{j}^{\prime }=
\begin{cases}
-P_{V}(x)_{j} & \text{if }x_{j}>P_{V}(x)_{j} \\ 
-\infty  & \text{if }x_{j}=P_{V}(x)_{j}.
\end{cases}
\label{au3}
\end{equation}

Consequently, $a^{\prime }h=\max_{j|x_{j}=P_{V}(x)_{j}}(h_{j}-x_{j})$ and $%
b^{\prime }h=\max_{j|x_{j}>P_{V}(x)_{j}}(h_{j}-P_{V}(x)_{j}\},$ whence by (%
\ref{half-reduc}) we obtain 
\begin{eqnarray*}
H&=&\{h\in \mathbb{R}_{\max }^{n}|a^{\prime }h\geq b^{\prime }h\} \\
&=&\{h\in \mathbb{R}_{\max }^{n}\mid
\max_{j|x_{j}=P_{V}(x)_{j}}(h_{j}-x_{j})\geq
\max_{j|x_{j}>P_{V}(x)_{j}}(h_{j}-(P_{V}(x)_{j}))\} \\
&=&H_{V,x}^{\prime }.\qed
\end{eqnarray*}
\renewcommand{\qed}{}
\end{proof}

\begin{remark}
\label{rassns}\emph{a) }In the above, since $x\notin V,$ we have $J^{c}\neq
\emptyset .$ Furthermore, we also have $J\neq \emptyset ,$ since otherwise $%
P_{V}(x)_{j}<x_{j}\;(j\in \lbrack n]),$ whence by (\ref{kefe}) we would
obtain $H_{V,x}^{\prime }=\emptyset .$

Note also that the coefficients $-x_{j}$ and $-P_{V}(x)_{j}$ in the
canonical form (\ref{kefe}) of $H_{V,x}^{\prime }$ depend on $V$ and $x,$
while the coefficients $a_{j}^{\prime },b_{j}^{\prime }$ in the canonical
form (\ref{half-reduc}) of (\ref{totaia}) don't.

\emph{b)} The assumption alone that all coordinates of $x$ are $>-\infty $
does not imply that each element $v$ of $V$ has all coordinates $>-\infty ,$
as shown e.g.\ by the subsemimodule $V=\{(-\infty ,v_{2})|v_{2}\in \mathbb{R}%
\}$ of $\mathbb{R}_{\max }^{2}.$

\emph{c)} Corollary \ref{cuniv} is a more precise form of \cite{GK}, Theorem
3.
\end{remark}

By (\ref{au}), (\ref{au3}) and the assumption that all $P_{V}(x)_{j}$ are $%
>-\infty ,$ we have $\func{Supp}\,(a^{\prime })=J$ and $\func{Supp}%
\,(b^{\prime })=J^{c},$ and hence in the situation of Corollary \ref{cuniv}
we always have 
\begin{equation}
\text{Supp}\,(a^{\prime })\cup \text{Supp}\,(b^{\prime })=J\cup J^{c}=[n].
\label{full}
\end{equation}

\begin{definition}
\label{dfull}We shall call the sets $H^{\prime }$ of the form (\ref
{half-reduc}) satisfying $\func{Supp}\,a^{\prime }\cap \func{Supp}%
\,b^{\prime }=\emptyset $ and (\ref{full}), \emph{half-spaces with finite
apex}.

Note that the sets of this form are exactly the ``tropical half-spaces''
studied in \cite{joswig04}, where the apex of the half-space (\ref
{half-reduc}) is defined as the vector$\,-(a^{\prime }\oplus b^{\prime }).$
\end{definition}

\begin{remark}
\label{rdiffer}In classical linear analysis, one first reduces the problem
of best approximation of elements $x$ by linear subspaces $V$ to the case of
suitable separating support half-spaces $H=H_{V,x}$ by showing for them the
equality of distances $d(x,V)=d(x,H)$ and the equality of elements of best
approximation in $V$ and $H$, then one solves the problems of best
approximation for general half-spaces $H,$ and this gives solutions also for
the problems of best approximation by the linear subspaces $V$. In the case
of best approximation of $x$ by elements of subsemimodules $V$ of $\mathbb{R}%
_{\max }^{n}$ such that all coordinates of $P_{V}(x)$ (and hence also of $x)$
are $>-\infty ,$ in order to apply such a method one needs to use closed
half-spaces with finite apex, as shown by Corollary \ref{cuniv}.
\end{remark}

The following immediate consequence of Corollary \ref{cuniv} shows that the 
\emph{sectors} of $H^{\prime }$, as defined in~\cite{joswig04} are readily
obtained from the previous representation, and that the apex of $H^{\prime}$
is precisely $P_V(x)$.

\begin{corollary}
Let $x,V$ and $H^{\prime }$ be as in Corollary~\ref{cuniv}. Then, the apex
of the half-space $H^{\prime }$ is $P_{V}(x)$, and $H^{\prime }$ is the
union of the sectors 
\begin{equation*}
H_{i}^{\prime }:=\{x\in \mathbb{R}_{\max }^{n}\mid h_{i}-(P_{V}(x))_{i}\geq
\max_{j\in \lbrack n]\setminus \{i\}}(h_{j}-(P_{V}(x))_{j})\}\quad \forall
i\in I\enspace.
\end{equation*}
\qed
\end{corollary}

In the above the term ``closed half-space'' was introduced because of the
analogy with the classical closed half-spaces $\{x\in \mathbb{R}^{n}|\Phi
(x)\leq c\}$ of $\mathbb{R}^{n}$, where $\Phi \in (\mathbb{R}^{n})^{\ast
},c\in \mathbb{R}.$ However, note that there is an important difference
between the two cases. Namely, in the classical case of $\mathbb{R}^{n}$,
given a linear subspace $V$ of $\mathbb{R}^{n}$ and a point $x\notin V,$
there exists a separating closed half-space $H=H_{V,x}$ of $\mathbb{R}^{n}$
(i.e.\ such that $V\subseteq H,x\notin H),$ with the additional property $%
d(x,V)=d(x,H),$ but for any other separating closed half-space $H^{\prime
}\neq H\;(V\subset H^{\prime },x\notin H^{\prime })$ we must have $H^{\prime
}\subset H$ (strictly) and hence $d(x,H^{\prime })<d(x,H),$ because bd $%
H^{\prime }$ must be parallel to bd $H$ (these facts are well known and easy
to prove). However, this fact is no longer true in the case of closed
half-paces $H=V,\,H^{\prime }$ and outside points $x\notin H^{\prime }$ in $%
\mathbb{R}_{\max }^{n},$ as shown by Example \ref{evax} below, in which $%
H^{\prime }\supset H,\,\,H^{\prime }\neq H,\;d(x,H^{\prime })=d(x,H)$:

\begin{example}
\label{evax} Let 
\begin{eqnarray*}
H&=&V:=\{v\in \mathbb{R}_{\max }^{3}\mid v_{2}\geq v_{1}\} \\
&=&\{v\in \mathbb{R}_{\max }^{3}|(-\infty )v_{1}\oplus 0v_{2}\oplus (-\infty
)v_{3}\geq 0v_{1}\oplus (-\infty )v_{2}\oplus (-\infty )v_{3}\},\; \\
x&:=&(2,1,0)^{T}\notin V.
\end{eqnarray*}
Then $V$ is a subsemimodule (actually a half-space, but not with finite
apex), and 
\begin{eqnarray*}
P_{V}(x)&=&\max \{v\in V|(v_{1},v_{2},v_{3})^{T}\leq (2,1,0)^{T}\}=(1,1,0)^{T},
\\
J&=&\{j|x_{j}=P_{V}(x)_{j}\}=\{2,3\},J^{c}=\{j|x_{j}>P_{V}(x)_{j}\}=\{1\},
\end{eqnarray*}
so $J\cup J^{c}=[3]$, and hence the universal separating closed half-space $%
H^{\prime }$ of (\ref{kefe}) has finite apex; in fact, 
\begin{eqnarray*}
H^{\prime }&=&H_{V,x}^{\prime }=\{h\in \mathbb{R}_{\max }^{3}|\max
(-x_{2}+h_{2},-x_{3}+h_{3})\geq -P_{V}(x)_{1}+h_{1}\} \\
&=&\{h|\max (-1+h_{2},0+h_{3})\geq -1+h_{1}\}=\{h|\max (h_{2},h_{3}+1)\geq
h_{1}\}.
\end{eqnarray*}
Furthermore, we have $d(x,V)=d(x,H^{\prime })$ and $H\subset H^{\prime }$
(strictly). This is illustrated in Figure~\ref{fig0a}, in which every
max-plus line through the origin (i.e.\ the set of multiples of a vector of $%
\mathbb{R}_{\max }^{3}$) is represented by its intersection point with a
hyperplane orthogonal to the main diagonal.
\end{example}

\begin{figure}[tbph]
\begin{picture}(0,0)\includegraphics{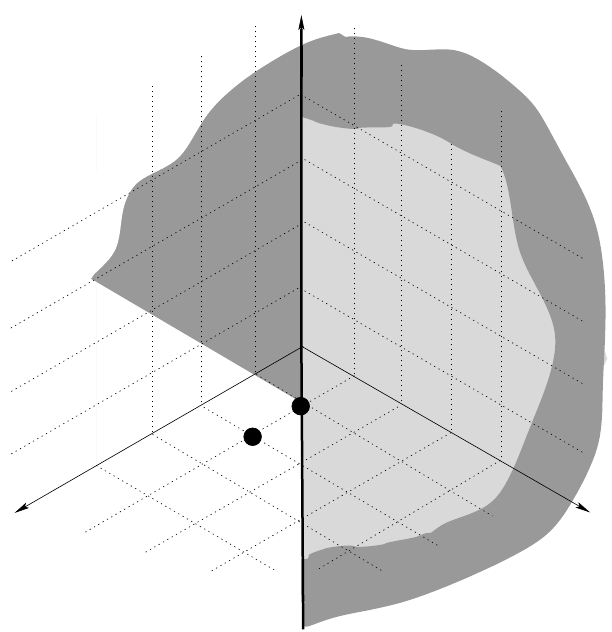}%
\end{picture}\setlength{\unitlength}{1579sp}\begingroup\makeatletter\ifx%
\SetFigFont\undefined\gdef\SetFigFont#1#2#3#4#5{\reset@font%
\fontsize{#1}{#2pt} \fontfamily{#3}\fontseries{#4}\fontshape{#5} \selectfont}%
\fi\endgroup%
\begin{picture}(7302,7558)(2386,-7391)
\put(2401,-5611){\makebox(0,0)[lb]{\smash{{\SetFigFont{10}{12.0}{\rmdefault}{\mddefault}{\updefault}{\color[rgb]{0,0,0}$h_1$}%
}}}}
\put(6001,-136){\makebox(0,0)[lb]{\smash{{\SetFigFont{10}{12.0}{\rmdefault}{\mddefault}{\updefault}{\color[rgb]{0,0,0}$h_3$}%
}}}}
\put(9376,-5686){\makebox(0,0)[lb]{\smash{{\SetFigFont{10}{12.0}{\rmdefault}{\mddefault}{\updefault}{\color[rgb]{0,0,0}$h_2$}%
}}}}
\put(5101,-4883){\makebox(0,0)[lb]{\smash{{\SetFigFont{10}{12.0}{\rmdefault}{\mddefault}{\updefault}{\color[rgb]{0,0,0}$x$}%
}}}}
\put(7126,-2536){\makebox(0,0)[lb]{\smash{{\SetFigFont{10}{12.0}{\rmdefault}{\mddefault}{\updefault}{\color[rgb]{0,0,0}$H$}%
}}}}
\put(4726,-2761){\makebox(0,0)[lb]{\smash{{\SetFigFont{10}{12.0}{\rmdefault}{\mddefault}{\updefault}{\color[rgb]{0,0,0}$H^{\prime}$}%
}}}}
\put(6076,-4711){\makebox(0,0)[lb]{\smash{{\SetFigFont{10}{12.0}{\rmdefault}{\mddefault}{\updefault}{\color[rgb]{0,0,0}$P_H(x)$}%
}}}}
\end{picture}
\caption{The half-space $H=\{h\in \mathbb{R}_{\max }^{3}\mid h_{2}\geq
h_{1}\}$ (light gray). The universal separating closed half-space $H^{\prime
}$ with apex $P_{H}(x)$ (dark gray), see Example~\ref{evax}.}
\label{fig0a}
\end{figure}
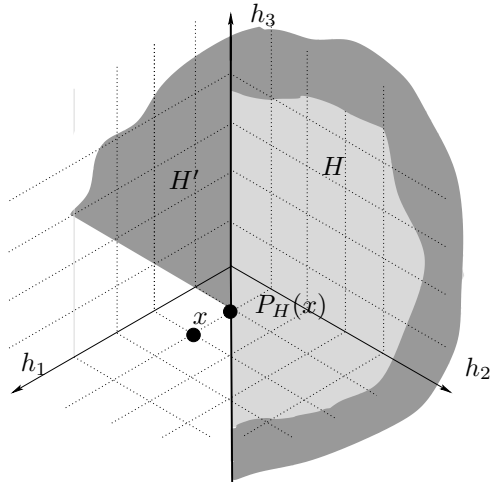

\section{The elements of best approximation by closed half-spaces}
\label{s07}

By the above results, the problem of best approximation by subsemimodules of 
$\mathbb{R}_{\max }^{n}$ can be reduced to that of best approximation by
closed half-spaces with finite apex. In the present section, more generally,
we give characterizations of the elements of best approximation by arbitrary
closed half-spaces in $\mathbb{R}_{\max }^{n}$ (that are not assumed to have
finite apex). If $a\in (\mathbb{R}_{\max }^{n})^{\ast }$ is a row vector and 
$x\in \mathbb{R}_{\max }^{n}$ a column vector, we define 
\begin{equation}
\func{Argmax}(a,x):=\{i\in \lbrack n]\mid a_{i}x_{i}=ax\}\enspace,
\label{argmax}
\end{equation}
which is always a nonempty set. The following is clear: 
\begin{equation}
ax\neq -\infty \;\Rightarrow \;\func{Argmax}(a,x)\subset \func{Supp}\,a\cap 
\func{Supp}\,x\enspace. \,  \label{argmax-supp}
\end{equation}

The next theorem gives an analytic characterization of the set of elements
of best approximation.

\begin{theorem}
\label{th-argmax} Let $a,b\in (\mathbb{R}_{\max }^{n})^{\ast }$ be row
vectors, $H$ the closed half-space (\ref{clohalfsp2}), and assume that the
sets 
\begin{equation}
I:=\func{Supp}\,a,\;\;J:=\func{Supp}\,b,  \label{cemai0}
\end{equation}
satisfy $I\cap J=\emptyset $ and $J\neq \emptyset $ $(b\neq -\infty ).$
Furthermore, let $x\in \mathbb{R}_{\max }^{n},x\not\in H$ be such that $%
d(x,H)<+\infty $. For an element $h\in \mathbb{R}_{\max }^{n}$ the following
assertions are equivalent :

\begin{enumerate}
\item  \label{best1} $h$ is a best approximation of $x$ in the closed
half-space $H$;

\item  \label{best2} $ah\geq bh\neq -\infty $ and 
\begin{equation}
x(bx)^{-1}(ah)\leq h\leq x(ax)^{-1}(bh);  \label{bla}
\end{equation}

\item  \label{best3} There exist $\lambda \neq -\infty $ and $i\in \func{%
Argmax}(a,x)$ such that the following conditions hold: 
\begin{eqnarray}
a_{i}h_{i}&=&\lambda \enspace,  \label{th-argmax5} \\
b_{j}h_{j}&=&\lambda \enspace,\quad \forall j\in \func{Argmax}(b,x)\enspace,
\label{cemai1} \\
x_{k}(bx)^{-1}\lambda &\leq& h_{k}\leq \big(P_{H}(x)\big)_{k}(ax)^{-1}\lambda
,\quad   \label{cemai2} \\
&&\qquad \forall k\in \lbrack n]\setminus (\func{Argmax}(b,x)\cup \{i\})%
\enspace;  \notag
\end{eqnarray}
moreover, in this case $\lambda $ is unique, namely $\lambda =ah=bh$.
\end{enumerate}
\end{theorem}

\begin{proof}
By Theorem~\ref{th1new} and our assumption, $d(x,H)={ax}\backslash {%
bx<+\infty }$, and so $ax\neq -\infty .$ Furthermore, since $x\not\in H$, we
have $bx>ax$, and in particular $bx\neq -\infty $.

\ref{best1}$^{\circ }\implies$\ref{best2}$^{\circ }$. Let $h$ be a best
approximation of $x$ in $H$, that is, $h\in H$ (so $ah\geq bh)$ and $%
d(x,h)=d(x,H)$, which is equivalent to the condition $\delta (x,h)\geq
\delta (x,H),$ that is, 
\begin{equation}
(x\backslash h)(h\backslash x)\geq {ax}(bx)^{-1}\enspace.  \label{equivopt}
\end{equation}
Since $d(x,h)=d(x,H)<+\infty $, $x$ and $h$ must have the same support (by
Lemma~\ref{l1}). Then, since $ax\neq -\infty $ we deduce that $ah\neq
-\infty $ (indeed, there is at least one index $i$ such that $%
a_{i}x_{i}=ax\neq -\infty $, and so $x_{i}\neq -\infty $; hence, since $x$
and $h$ have the same support, $h_{i}\neq -\infty $, and so $ah\geq
a_{i}h_{i}\neq -\infty $). Similarly, we deduce from $bx\neq -\infty $, that 
$bh\neq -\infty $. Furthermore,~(\ref{equivopt}) implies that $x\backslash
h\geq (h\backslash x)^{-1}(ax)(bx)^{-1}$ or equivalently (see (\ref{resid2}%
)), 
\begin{equation}
h\geq x(h\backslash x)^{-1}(ax)(bx)^{-1}\enspace.  \label{e:h2-1}
\end{equation}
Similarly, from (\ref{equivopt}) one also obtains 
\begin{equation}
x(x\backslash h)(ax)^{-1}(bx)\geq h\enspace.  \label{e:h2-3}
\end{equation}

Since $h\lambda \leq x$ implies $ah\lambda \leq ax$, it follows that $%
h\backslash x\leq \sup \{\lambda |ah\lambda \leq ax\}=(ah)\backslash
(ax)=(ah)^{-1}(ax)$, whence $(h\backslash x)^{-1}\geq (ax)^{-1}ah.$ Using
this inequality in~(\ref{e:h2-1}), we get $h\geq x(ah)(bx)^{-1}$, which is
the first inequality in~(\ref{bla}). Furthermore, $x\backslash h\leq
(bx)\backslash (bh)=(bx)^{-1}(bh)$ and together with (\ref{e:h2-3}), this
implies that $h\leq x(bh)(ax)^{-1}$, which is the second inequality in~(\ref
{bla}). This completes the proof of~the implication \ref{best1}$^{\circ }%
\implies$\ref{best2}$^{\circ }$.

\ref{best2}$^{\circ }\implies$\ref{best1}$^{\circ }$. Let $h$ be as in~\ref
{best2}$^{\circ }$. Then by $ah\geq bh$ we have $h\in H$. Using (\ref{resid2}%
), from the first inequality in~(\ref{bla}) we obtain that $x\backslash
h\geq (bx)^{-1}(ah)$. By the second inequality in~(\ref{bla}), and the fact
that $bh\neq -\infty $, we obtain that $h(bh)^{-1}(ax)\leq x$, which
implies, using (\ref{resid2}), that $h\backslash x\geq (bh)^{-1}(ax)$. Hence 
$(h\backslash x)(x\backslash h)\geq (bx)^{-1}(ah)(bh)^{-1}(ax)$ and since $%
ah\geq bh$, we obtain $(h\backslash x)(x\backslash h)\geq (bx)^{-1}(ax)$,
that is,~(\ref{equivopt}), which itself is equivalent to the condition $%
d(x,h)=d(x,H)$.

\ref{best2}$^{\circ }\implies$\ref{best3}$^{\circ }$. Let $h$ be as in~\ref
{best2}$^{\circ }$, set $\lambda =bh$, and pick some $i\in \func{Argmax}(a,h)
$ (the latter set is necessarily nonempty). We shall see later that 
\begin{equation}
\func{Argmax}(a,h)\subset \func{Argmax}(a,x)\enspace,  \label{e-actual}
\end{equation}
so that $i\in \func{Argmax}(a,x)$ as requested in \ref{best3}$^{\circ }$.

By~\ref{best2}$^{\circ }$, we must have $\lambda =bh\in \mathbb{R}$.
Multiplying the first inequality of~(\ref{bla}) by $b$, or the second one by 
$a$, we deduce that $ah\leq bh$, and since $ah\geq bh$ also holds by~\ref
{best2}$^{\circ }$, we get that $ah=bh=\lambda $. Consequently, since by our
choice $i\in \func{Argmax}(a,h),$ we have $a_{i}h_{i}=\lambda $. 

Using the fact that $ah=bh=\lambda $, we deduce from~\eqref{bla} that 
\begin{equation}
x_{k}(bx)^{-1}\lambda \leq h_{k}\leq x_{k}(ax)^{-1}\lambda ,\qquad \forall
k\in \lbrack n]\enspace.\label{e-newp}
\end{equation}
Furthermore, since $b_{k}h_{k}\leq bh=\lambda $, we deduce $h_{k}\leq
b_{k}\backslash \lambda =(b_{k}\backslash (ax))(ax)^{-1}\lambda $, for all $%
k\in \lbrack n]$. This, together with the second inequality in~(\ref{e-newp})
 and formula~%
\eqref{e0} for $P_{H}(x)$, implies 
\begin{equation*}
h_{k}\leq \big(x_{k}\wedge b_{k}\backslash (ax)\big)(ax)^{-1}\lambda =\big(%
P_{H}(x)\big)_{k}(ax)^{-1}\lambda \enspace.
\end{equation*}
Together with the first inequality in~\eqref{cemai2}, this establishes the
inequalities~\eqref{cemai2} for all $k\in \lbrack n]$, and a fortiori for
all $k\in \lbrack n]\setminus (\func{Argmax}(b,x)\cup \{i\})$. %

Now we show~\eqref{e-actual}. By the second part of (\ref{bla}) we have $%
a_{k}h_{k}\leq a_{k}x_{k}(ax)^{-1}\lambda $ for all $k\in \lbrack n],$ and
hence for any $k$ such that $a_{k}x_{k}<ax,$ we have $a_{k}h_{k}\leq
a_{k}x_{k}(ax)^{-1}\lambda <\lambda =ah,$ whence $k\notin \func{Argmax}(a,h),
$ which shows~\eqref{e-actual}. Since we already proved that $%
a_{i}h_{i}=\lambda $, we deduce~\eqref{th-argmax5}. 

Finally, if $j\in \func{Argmax}(b,x),$ that is, $b_{j}x_{j}=bx,$ then $%
\lambda =bx(bx)^{-1}\lambda =b_{j}x_{j}(bx)^{-1}\lambda \leq b_{j}h_{j}\leq
bh=\lambda $ (where the penultimate inequality is obtained by multiplying by 
$b_{j}$ the first inequality of (\ref{cemai2}) for $k=j),$ whence we obtain (%
\ref{cemai1}).

\ref{best3}$^{\circ }\implies$\ref{best2}$^{\circ }$. Let $h$, $i$ and $%
\lambda $ be as in~\ref{best3}$^{\circ }$.

We claim that the inequalities in~\eqref{cemai2} are valid for all $k\in
\lbrack n]$.

Indeed, since $i\in \func{Argmax}(a,x)$, then by $a_{i}x_{i}=ax<bx$ we have
\[
a_{i}x_{i}(bx)^{-1}\lambda <bx(bx)^{-1}\lambda =\lambda
=a_{i}h_{i}=a_{i}x_{i}(ax)^{-1}\lambda
\enspace .
\]
Moreover, since $%
a_{i}h_{i}=\lambda \in \mathbb{R}$, we have $a_{i}\neq -\infty $, and so $%
x_{i}(bx)^{-1}\lambda \leq h_{i}\leq x_{i}(ax)^{-1}\lambda $. Since $i\in
I\subset J^{c}$ by assumption, hence $b_{i}\backslash (ax)=+\infty $ (see (%
\ref{bepel})), and so, 
\begin{equation*}
h_{i}\leq (x_{i}\wedge b_{i}\backslash (ax))(ax)^{-1}\lambda =\big(P_{H}(x)%
\big)_{i}(ax)^{-1}\lambda \enspace.
\end{equation*}
We deduce that~(\ref{cemai2}) is valid for $k=i$.

Similarly, if $k\in \func{Argmax}(b,x)$, then by $b_{k}x_{k}=bx$ and $ax<bx,$
we have $b_{k}x_{k}(bx)^{-1}\lambda =\lambda <b_{k}x_{k}(ax)^{-1}\lambda $,
where by~(\ref{cemai1}) we have $\lambda =b_{k}h_{k}\in \mathbb{R}$; whence $%
b_{k}\neq -\infty $. Consequently, 
\begin{equation}
x_{k}(bx)^{-1}\lambda \leq h_{k}\leq x_{k}(ax)^{-1}\lambda \enspace.
\label{e-sandw}
\end{equation}
Moreover, by~(\ref{cemai1}), and the fact that $b_{k}\neq -\infty $, we get
that $h_{k}=b_{k}\backslash \lambda $ for all $k\in \func{Argmax}(b,x)$,
hence, 
\begin{equation*}
h_{k}=(b_{k}\backslash (ax))(ax)^{-1}\lambda \enspace.
\end{equation*}
This, together with the second inequality in~\eqref{e-sandw}, shows that 
\begin{equation*}
h_{k}=\big(x_{k}\wedge (b_{k}\backslash (ax))\big)(ax)^{-1}\lambda =\big(%
P_{H}(x)\big)_{k}(ax)^{-1}\lambda 
\end{equation*}
and so, (\ref{cemai2}) is valid for these $k$, which proves the claim.

Multiplying the second inequality in~(\ref{cemai2}) by $a_{k}$, and using $%
P_{H}(x)\leq x$, we obtain that $a_{k}h_{k}\leq a_{k}x_{k}(ax)^{-1}\lambda
\leq \lambda $ for all $k\in \lbrack n],$ and using~(\ref{th-argmax5}), we
get that $ah=\lambda $. Similarly, multiplying the second inequality in~(\ref
{cemai2}) by $b_{k}$, and using again $P_{H}(x)\leq x$, we obtain that $%
b_{k}h_{k}\leq b_{k}x_{k}(ax)^{-1}\lambda \leq bx(ax)^{-1}\lambda \leq
\lambda $ for all $k\in \lbrack n]$, and using~(\ref{cemai1}), we get that $%
bh=\lambda =ah$. This equality, together with~(\ref{cemai2}), which is valid
for all $k\in \lbrack n]$, imply~\ref{best2}$^{\circ }$.
\end{proof}

\begin{remark}
We observed in the proof of Theorem~\ref{th-argmax} that if $h$ is an
element of best approximation of $x$, the inequality~\eqref{cemai2} actually
holds for all $k\in \lbrack n]$. It follows that 
\begin{equation*}
P_{H}(x)(bx)^{-1}\lambda \leq x(bx)^{-1}\lambda \leq h\leq
P_{H}(x)(ax)^{-1}\lambda .
\end{equation*}
By comparing $h$ with the extreme terms in the above inequalities, and using
the characterization~\eqref{e-d-hilbert} of Hilbert's projective distance,
we deduce that 
\begin{equation*}
d(h,P_{H}(x))\leq (ax)^{-1}(bx)=d(x,H)
\end{equation*}
so that $h$ lies in the intersection of two balls of radius $d(x,H)$ in
Hilbert's projective metric, one being centered at the point $x$, the other
being centered at the point $P_{H}(x)$.
\end{remark}

\begin{remark}
\label{remark-face} One can give a geometric interpretation of the
conditions of Theorem \ref{th-argmax} in terms of faces of the ball with
center $x$ and radius $d(x,H).$ Indeed, let us fix some index $i\in \func{%
Argmax}(a,x)$, and let $F_{i}$ denote the set of vectors $h$ satisfying the
conditions \eqref{th-argmax5}, \eqref{cemai1}, \eqref{cemai2} of Theorem~\ref
{th-argmax}. Then the conditions that $a_{i}h_{i}=b_{j}h_{j}$ for all $j\in 
\func{Argmax}(b,x)$, together with $ax=a_{i}x$ and $bx=b_{j}x$, lead to $%
h_{i}h_{j}^{-1}=b_{j}a_{i}^{-1}=x_{i}x_{j}^{-1}(bx)(ax)^{-1}$. This can be
rewritten with the usual linear algebraic notation, as 
\begin{equation*}
h_{i}-h_{j}=x_{i}-x_{j}+d(x,H),\qquad \forall j\in \func{Argmax}(b,x).
\end{equation*}
Thus, if $p$ is the cardinality of $\func{Argmax}(b,x)$, we see that $h$
satisfies $p$ of the inequalities defining the facets of the ball of radius $%
d(x,H)$ in Hilbert metric, centered at the point $x$ (the ball in Hilbert
metric is a polyhedron in the usual sense, and so the standard notions of
faces and facets -maximal faces-, see~\cite{ziegler}, apply to it).
Therefore, the set $F_{i}$ consisting of these vectors $h$ lies in a $n-p$
dimensional face of this ball, and Theorem~\ref{th-argmax} gives a
disjunctive representation of the set of elements of best approximation, as
the union of the sets $F_{i}$ with $i\in \func{Argmax}(i,x)$. Note that the
inequalities~\eqref{cemai2} indicate that $F_{i}$ may be a strict subset of
a face of the latter ball, as illustrated in Figure~\ref{fig-projection2}
below (right).
\end{remark}

Let us give some geometric interpretations of best approximation by closed
half-spaces in simple particular cases.

\begin{example}
\label{rulter} Let $n=3,$ and 
\begin{equation*}
H:=\{h\in \mathbb{R}_{\max }^{3}\mid h_{2}\geq h_{1}\}=\{h\in \mathbb{R}%
_{\max }^{3}|ah\geq bh\},
\end{equation*}
where $a=(-\infty ,0,-\infty ),b=(0,-\infty ,-\infty ),$ and let $%
x_{1}>x_{2}>x_{3}$. Then, with the notations of the proof of Theorem \ref
{th-argmax}, we have $\func{Argmax}(a,x)=I=\{2\}$ and $\func{Argmax}%
(b,x)=J=\{1\}$, so necessarily $i=2$ and $h=(h_{1},h_{2},h_{3})^{T}$ is an
element of best approximation of $x$ in $H$ if and only if there exists $%
\lambda \in \mathbb{R}$ \ such that 
\begin{equation*}
h_{2}=h_{1}=\lambda ,\qquad x_{3}-x_{1}+\lambda \leq h_{3}\leq
x_{3}-x_{2}+\lambda .
\end{equation*}
The half-space $H$ was already represented in Figure~\ref{fig0a}, Assume now
that $x=(2,1,0)^{T}$, so that, as noted in Example~\ref{evax}, $%
P_{H}(x)=(1,1,0)^{T}$. %
By Remark~\ref{remark-face}, the set of elements of best approximation is
the set $F_{2}$, which lies in a two dimensional face of a ball in Hilbert's
metric. This set is represented by a bold segment in Figure~\ref{fig0}.

\end{example}

\begin{figure}[tbph]
\begin{picture}(0,0)\includegraphics{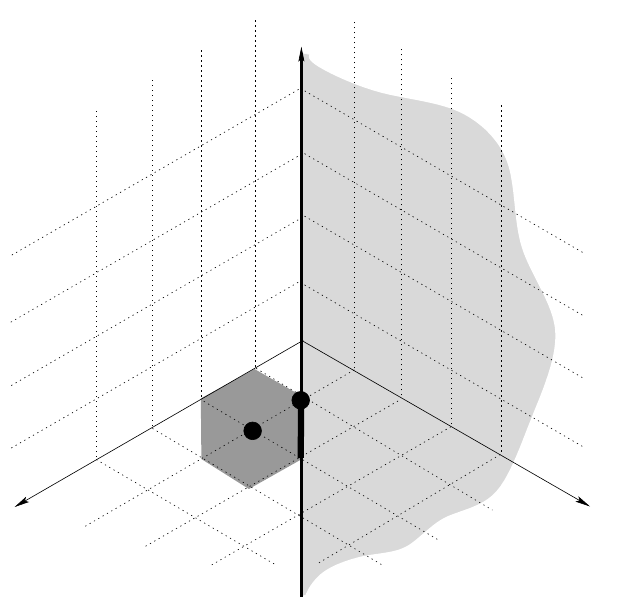}%
\end{picture}\setlength{\unitlength}{1579sp}\begingroup\makeatletter\ifx%
\SetFigFont\undefined\gdef\SetFigFont#1#2#3#4#5{\reset@font%
\fontsize{#1}{#2pt} \fontfamily{#3}\fontseries{#4}\fontshape{#5} \selectfont}%
\fi\endgroup%
\begin{picture}(7638,7262)(2386,-7170)
\put(2401,-5611){\makebox(0,0)[lb]{\smash{{\SetFigFont{10}{12.0}{\rmdefault}{\mddefault}{\updefault}{\color[rgb]{0,0,0}$h_1$}%
}}}}
\put(5701,-211){\makebox(0,0)[lb]{\smash{{\SetFigFont{10}{12.0}{\rmdefault}{\mddefault}{\updefault}{\color[rgb]{0,0,0}$h_3$}%
}}}}
\put(9376,-5686){\makebox(0,0)[lb]{\smash{{\SetFigFont{10}{12.0}{\rmdefault}{\mddefault}{\updefault}{\color[rgb]{0,0,0}$h_2$}%
}}}}
\put(7126,-2536){\makebox(0,0)[lb]{\smash{{\SetFigFont{10}{12.0}{\rmdefault}{\mddefault}{\updefault}{\color[rgb]{0,0,0}$H$}%
}}}}
\put(6136,-4689){\makebox(0,0)[lb]{\smash{{\SetFigFont{10}{12.0}{\rmdefault}{\mddefault}{\updefault}{\color[rgb]{0,0,0}$P_H(x)$}%
}}}}
\put(5101,-4883){\makebox(0,0)[lb]{\smash{{\SetFigFont{10}{12.0}{\rmdefault}{\mddefault}{\updefault}{\color[rgb]{0,0,0}$x$}%
}}}}
\end{picture}
\caption{Illustration of Theorem~\ref{th-argmax} (see Example~\ref{rulter}).
The half-space $H=\{h\in \mathbb{R}_{\max }^{3}\mid h_{2}\geq h_{1}\}$
(light gray); the maximal open ball in Hilbert's metric centered at point $%
x=(2,1,0)^{T}$ and contained in the complement of $H$ (dark gray): the
projection $P_{H}(x)$ is visible at its boundary. The set of elements of
best approximation of $x$ is the bold segment. }
\label{fig0}
\end{figure}
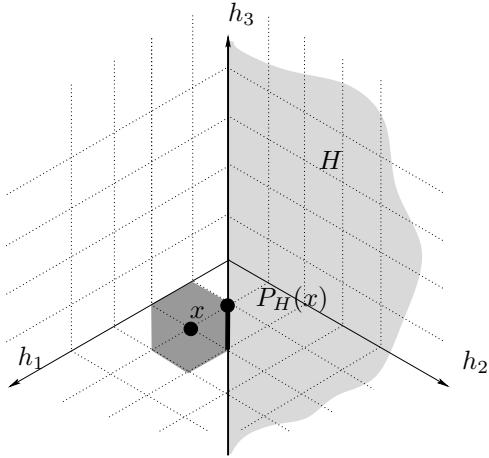

\begin{example}
\label{example-disjunctive} Consider now 
\begin{equation*}
H=\{h\in \mathbb{R}_{\max }^{3}\mid \max (h_{1},h_{3})\geq h_{2}\}
\end{equation*}
and $x=(0,1,0)^{T}$. Here, $a=(0,-\infty ,0)$ and $b=(-\infty ,0,-\infty )$.
We have $P_{H}(x)=(0,0,0)^{T}$, $d(x,P_{H}(x))=1$, and 
\begin{equation*}
\func{Argmax}(a,x)=\{1,3\},\qquad \func{Argmax}(b,x)=\{2\}\enspace.
\end{equation*}
Theorem~\ref{th-argmax} shows that the set of elements of best approximation
of $x$ is the union of the sets $F_{1}$ and $F_{3}$ defined in Remark~\ref
{remark-face}. Condition~\ref{best3}$^{\circ }$ of Theorem~\ref{th-argmax}
yields 
\begin{equation*}
F_{1}=\{h\in \mathbb{R}^{3}\mid h_{1}=h_{2},\;-1+h_{1}\leq h_{3}\leq h_{1}\}%
\enspace.
\end{equation*}
By symmetry, $F_{3}$ is obtained from $F_{1}$ by exchanging the variables $%
h_{1}$ and $h_{3}$. %
This is illustrated in Figure~\ref{fig-projection2} (left). 
\end{example}

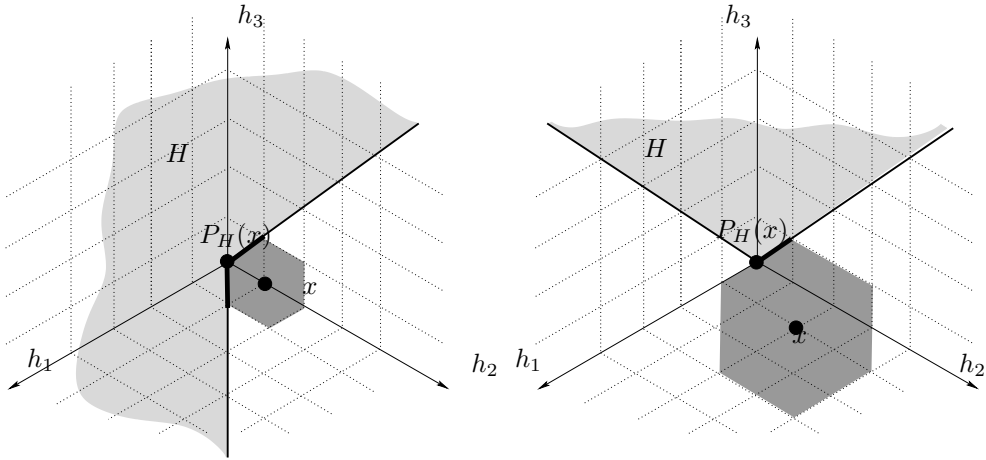
\begin{figure}[htbp]
\begin{tabular}{cc}
\begin{picture}(0,0)\includegraphics{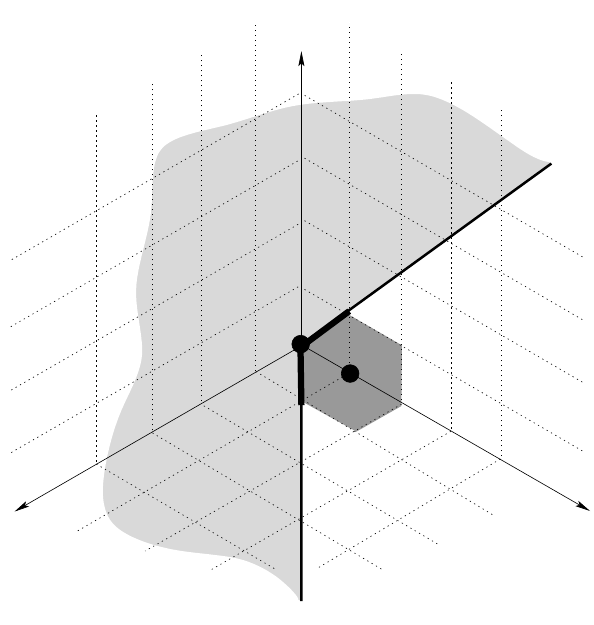}%
\end{picture}\setlength{\unitlength}{1579sp} \begingroup\makeatletter\ifx%
\SetFigFont\undefined\gdef\SetFigFont#1#2#3#4#5{\ \reset@font%
\fontsize{#1}{#2pt} \fontfamily{#3}\fontseries{#4}\fontshape{#5} \selectfont}%
\fi\endgroup%
\begin{picture}(7091,7490)(2386,-7322)
\put(2401,-5611){\makebox(0,0)[lb]{\smash{{\SetFigFont{10}{12.0}{\rmdefault}{\mddefault}{\updefault}{\color[rgb]{0,0,0}$h_1$}}}}} \put(5701,-211){\makebox(0,0)[lb]{\smash{{\SetFigFont{10}{12.0}{\rmdefault}{\mddefault}{\updefault}{\color[rgb]{0,0,0}$h_3$}}}}} \put(9376,-5686){\makebox(0,0)[lb]{\smash{{\SetFigFont{10}{12.0}{\rmdefault}{\mddefault}{\updefault}{\color[rgb]{0,0,0}$h_2$}}}}} \put(4576,-2386){\makebox(0,0)[lb]{\smash{{\SetFigFont{10}{12.0}{\rmdefault}{\mddefault}{\updefault}{\color[rgb]{0,0,0}$H$}}}}} \put(5106,-3671){\makebox(0,0)[lb]{\smash{{\SetFigFont{10}{12.0}{\rmdefault}{\mddefault}{\updefault}{\color[rgb]{0,0,0}$P_H(x)$}}}}} \put(6721,-4471){\makebox(0,0)[lb]{\smash{{\SetFigFont{10}{12.0}{\rmdefault}{\mddefault}{\updefault}{\color[rgb]{0,0,0}$x$}}}}} \end{picture}
& \begin{picture}(0,0)\includegraphics{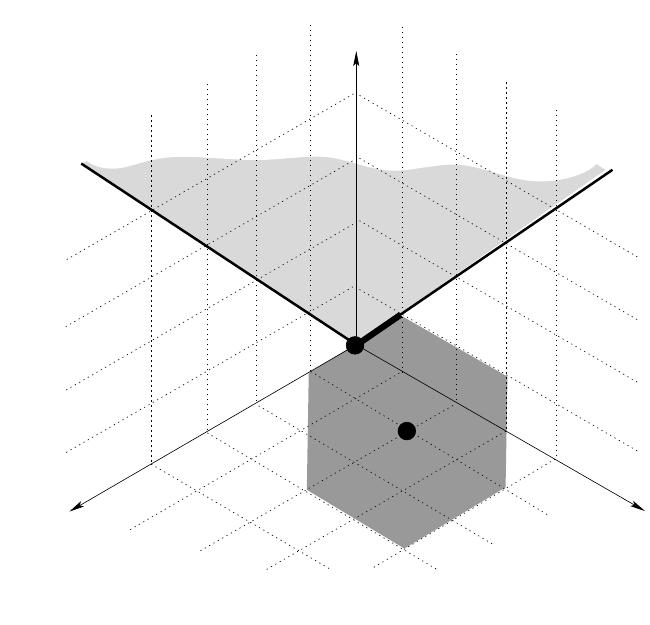}%
\end{picture}\setlength{\unitlength}{1579sp} \begingroup\makeatletter\ifx%
\SetFigFont\undefined\gdef\SetFigFont#1#2#3#4#5{\ \reset@font%
\fontsize{#1}{#2pt} \fontfamily{#3}\fontseries{#4}\fontshape{#5} \selectfont}%
\fi\endgroup%
\begin{picture}(7091,7490)(2386,-7322)
\put(2401,-5611){\makebox(0,0)[lb]{\smash{{\SetFigFont{10}{12.0}{\rmdefault}{\mddefault}{\updefault}{\color[rgb]{0,0,0}$h_1$}}}}} \put(5701,-211){\makebox(0,0)[lb]{\smash{{\SetFigFont{10}{12.0}{\rmdefault}{\mddefault}{\updefault}{\color[rgb]{0,0,0}$h_3$}}}}} \put(9376,-5686){\makebox(0,0)[lb]{\smash{{\SetFigFont{10}{12.0}{\rmdefault}{\mddefault}{\updefault}{\color[rgb]{0,0,0}$h_2$}}}}} \put(4426,-2311){\makebox(0,0)[lb]{\smash{{\SetFigFont{10}{12.0}{\rmdefault}{\mddefault}{\updefault}{\color[rgb]{0,0,0}$H$}}}}} \put(5551,-3586){\makebox(0,0)[lb]{\smash{{\SetFigFont{10}{12.0}{\rmdefault}{\mddefault}{\updefault}{\color[rgb]{0,0,0}$P_H(x)$}}}}} \put(6751,-5236){\makebox(0,0)[lb]{\smash{{\SetFigFont{10}{12.0}{\rmdefault}{\mddefault}{\updefault}{\color[rgb]{0,0,0}$x$}}}}} \end{picture}
\end{tabular}
\caption{Left. A set of elements of best approximation of a disjunctive
nature (Example~\ref{example-disjunctive}). Right. The set of elements of
best approximation may be a strict subset of a face of a Hilbert ball
(Remark~\ref{remark-face} and Example~\ref{ex-subface}).}
\label{fig-projection2}
\end{figure}

\begin{example}
\label{ex-subface} Let 
\begin{equation*}
H:=\{x\in \mathbb{R}_{\max }^{3}\mid h_{3}\geq \max (h_{1},h_{2})\}\enspace.
\end{equation*}
and $x=(1,2,0)^{T}$. It can be checked that $P_{H}(x)=(0,0,0)^{T}$, and that
the set of elements of best approximation of $x$ is a strict subset of a
face of the ball of radius $d(x,P_{H}(x))=2$, centered at $x$, see Figure~%
\ref{fig-projection2}.
\end{example}

\section{The cyclic projection algorithm to solve max-plus linear systems}
\label{s08}

The max-plus analogue, studied in~\cite{GS}, of the classical cyclic
projection technique allows one to compute the canonical projection of a
vector $u\in \mathbb{R}_{\max }^{n}$ onto a subsemimodule 
\begin{equation}
V:=V_{1}\cap \cdots \cap V_{p}  \label{inters}
\end{equation}
defined as the intersection of $p$ closed subsemimodules by successively
projecting onto $V_{1},V_{2},\ldots ,V_{p},V_{1},\ldots $. The application
of this idea to the case of intersection of half-spaces, thanks to Theorem~%
\ref{prop1}, will lead us to a new algorithm to solve the system of
inequalities 
\begin{equation}
Ax\geq Bx  \label{ineq1}
\end{equation}
where $A,B$ are $p\times n$ matrices with entries in $\mathbb{R}_{\max }$.

Let us first explain how the method of \cite{GS} leads to a general
algorithm. Formally, starting from an arbitrary finite vector $\xi ^{0}=u,$
we compute the sequence 
\begin{equation}
\xi ^{k+1}=P_{V_{(k+1\,\func{mod}\,p)}}(\xi ^{k})\qquad \forall k\geq 0%
\enspace,  \label{ineq2}
\end{equation}
where $(l\,\func{mod}\,p)$ denotes the unique number belonging to the set $%
[p]=\{1,\dots,p\}$ congruent to $l$ modulo $p$ and $P_{V_{j}}$ denotes the
canonical projection onto $V_{j}.$

\begin{theorem}
\label{th-cyclic} The sequence $\xi ^{k}$ generated by the cyclic projection
algorithm is non-increasing and converges to $P_{V}(u)$.
\end{theorem}

\begin{proof}
Since $P_{V_{k}}(x)\leq x$ holds for all $x$ and for all $k$, we have 
\begin{equation*}
\xi ^{k+1}=P_{V_{(k+1\,\func{mod}\,p)}}(\xi ^{k})\leq \xi ^{k},\quad \quad
\forall k\geq 0,\;
\end{equation*}
so the sequence $\xi ^{k}$ is non-increasing. We prove by induction that 
\begin{equation*}
\xi ^{k}\geq P_{V}(u),\quad \quad \forall k\geq 0.
\end{equation*}
For $k=0$, this follows from $u\geq P_{V}(u)$. Assume now that $\xi ^{k}\geq
P_{V}(u)$. Since $V\subset V_{(k+1\,\func{mod}\,p)}$ and since $P_{V}(.)$ is
a monotone idempotent function, we have 
\begin{equation*}
\xi ^{k+1}=P_{V_{(k+1\,\func{mod}\,p)}}(\xi ^{k})\geq P_{V}(\xi ^{k})\geq
P_{V}(P_{V}(u))=P_{V}(u),
\end{equation*}
which concludes the proof by induction. \ Hence, the non-increasing sequence 
$\xi ^{k}$ must have a limit, $\xi ^{\infty }$, such that 
\begin{equation*}
u\geq \xi ^{\infty }\geq P_{V}(u).
\end{equation*}
Consequently, again since $P_{V}(.)$ is a monotone idempotent function, 
\begin{equation*}
P_{V}(u)\geq P_{V}(\xi ^{\infty })\geq P_{V}(P_{V}(u))=P_{V}(u),
\end{equation*}
whence $P_{V}(u)=$ $P_{V}(\xi ^{\infty }).$ Therefore, in order to show that
the equality $\xi ^{\infty }=P_{V}(u)$ holds, it suffices to show that $\xi
^{\infty }=P_{V}(\xi ^{\infty }),$ i.e., that $\xi ^{\infty }\in V.$

Observe that for all $m\in \lbrack p]$, $\xi ^{\infty }$ is a limit of the
subsequence of $\xi ^{k}$ obtained by taking all the indices $k$ such that $%
(k+1\,\func{mod}\,p)=m$. Since $V_{m}$ is closed, it follows that $\xi
^{\infty }\in V_{m}$. Since this holds for all $m\in \lbrack p]$, we deduce
that $\xi ^{\infty }\in V$.
\end{proof}

The following is an immediate corollary.

\begin{corollary}
The intersection $V=V_{1}\cap \cdots \cap V_{p}$ is not reduced to the $%
-\infty $ vector if and only if the cyclic projection algorithm, initialized
by taking $\xi ^{0}$ to be any finite vector $u$, converges to a non-$%
(-\infty )$ vector $\xi ^{\infty }$ (and then this vector is precisely $\xi
^{\infty }=P_{V}(u)\in V$).\qed
\end{corollary}

Applying this algorithm to the case of intersection of half-spaces, and
using Theorem~\ref{prop1}, we obtain the following algorithm to solve the
system of inequalities (\ref{ineq1}), where $A,B$ are $p\times n$ matrices
with entries in $\mathbb{R}_{\max }$. We have 
\begin{equation}
V=H_{1}\cap \cdots\cap H_{p},  \label{inters2}
\end{equation}
where $H_{j}$ is the half-space 
\begin{equation}
H_{j}:=\{x\in \mathbb{R}_{\max }^{n}\mid A_{j}x\geq B_{j}x\}\quad \quad
\forall j\in[p],  \label{ineq4}
\end{equation}
with $A_{j}:=(A_{j1},\dots,A_{jn})$ and $B_{j}:=(B_{j1},\dots,B_{jn})$
denoting the $j$th rows of $A$ and $B$, respectively. Hence, by Theorem \ref
{prop1}, we obtain 
\begin{equation}
P_{H_{j}}(x)=x\wedge (B_{j}\backslash A_{j}(x))\quad \quad \forall x\in 
\mathbb{R}_{\max }^{n}, \forall j\in[p],  \label{ineq5}
\end{equation}
and thus, in particular, %
\begin{equation*}
\xi ^{k+1}=P_{H_{j}}(\xi ^{k})=\xi ^{k}\wedge (B_{j}\backslash A_{j}(\xi
^{k}))\quad \quad \forall k\geq 0 , \quad j:=(k+1\,\func{mod}\,p) \enspace .
\end{equation*}

Componentwise this means, by (\ref{phx1}) and (\ref{bepel}), that for each $%
k=0,1,\dots$ we have 
\begin{equation}
\xi _{i}^{k+1}=P_{H_{j}}(\xi _{i}^{k})=\xi _{i}^{k}\wedge (B_{ji}\backslash
(A_{j}(\xi ^{k}))\quad \quad \forall i\in[n]\enspace,  \label{ineq6}
\end{equation}
where again $j=(k+1\,\func{mod}\,p)$ and $A_{j}(\xi ^{k})=\oplus
_{i=1}^{n}A_{ji}\xi _{i}^{k}.$

An alternative method to the cyclic projection technique is the following 
\emph{power algorithm}, which is based on the observation that $Ax\geq Bx$
if and only if $x=B^{\sharp }Ax\wedge x$ (see (\ref{resi4})). The latter
fixed point problem can be solved by the iterative scheme 
\begin{equation}
\eta^{0}=u, \qquad \eta ^{k+1}=B^{\#}A\eta ^{k}\wedge \eta ^{k}, \qquad
\forall k\geq 0 \enspace,  \label{power}
\end{equation}
with $B^{\#}$ of (\ref{resi1}). This method may be thought of as a
generalization of the \emph{alternated projection algorithm} of Butkovi\v{c}
and Cuninghame-Green~\cite{CGB} which concerns the special case of the
linear system $Ay=Bz$ (the latter can be reduced to the former by setting $%
x=(y,z)$ and suitably extending the matrices $A$ and $B$).

In order to compare the power algorithm with the cyclic projection algorithm
we shall need the following ``sandwich theorem'':

\begin{theorem}
\label{tsandw}Consider the linear system $Ax\geq Bx$. Let $\eta ^{k}$ and $%
\xi ^{k}$ denote the sequences generated by the power and cyclic projection
algorithms, respectively, initialized with the same initial condition $u$.
Then 
\begin{equation}
P_{V}(u)\leq \xi ^{pk}\leq \eta ^{k}\quad \quad \forall k\geq 0\enspace.
\label{san}
\end{equation}
\end{theorem}

\begin{proof}
We shall use that the operator $\eta \rightarrow (B^{\#}A\eta )_{j}\wedge
\eta _{j}$ is monotone.

The first inequality follows from Theorem~\ref{th-cyclic}.

We now show that $\xi ^{p}\leq \eta ^{1}$. By (\ref{resi1}) we have 
\begin{equation*}
\eta _{j}^{1}=(B^{\#}Au)_{j}\wedge u_{j}=\wedge _{l=1}^{p}(-B_{lj}+^{\prime
}(A_{l}u))\wedge u_{j}=(-B_{ij}+^{\prime }(A_{i}u))\wedge u_{j},
\end{equation*}
for some $i\in \lbrack p]$. Hence, using that $\xi ^{k}$ is non-increasing, (%
\ref{ineq6}) for $k=0$ and (\ref{bepel}), it follows that 
\begin{equation*}
\xi _{j}^{p}\leq \xi _{j}^{i}=(B_{i}\backslash (A_{i}\xi ^{i-1}))_{j}\wedge
\xi _{j}^{i-1}\leq (B_{ij}^{-1}A_{i}u)\wedge u_{j}=\eta _{j}^{1}.
\end{equation*}

The inequality $\xi ^{pk}\leq \eta ^{k}$ is obtained by induction. For $k=1$
it is already proved. Assume now that it holds for $k$ replaced by $k-1$.
Then, by~(\ref{resi1}) we have 
\begin{eqnarray*}
\eta _{j}^{k}&=&(B^{\#}A\eta ^{k-1})_{j}\wedge \eta _{j}^{k-1}=\wedge
_{l=1}^{p}(-B_{lj}+^{\prime }(A_{l}\eta ^{k-1}))\wedge \eta _{j}^{k-1} \\
&=&(-B_{ij}+^{\prime }(A_{i}\eta ^{k-1}))\wedge \eta _{j}^{k-1},
\end{eqnarray*}
for some $i\in \lbrack p]$. Then using that that $\xi ^{k}$ is
non-increasing and 
\begin{equation*}
\xi ^{p(k-1)+i-1}\leq \xi ^{p(k-1)}\leq \eta ^{k-1},\quad \quad 
\end{equation*}
it follows that 
\begin{eqnarray*}
\xi _{j}^{pk}&\leq& \xi _{j}^{p(k-1)+i}=(B_{i}\backslash (A_{i}\xi
^{p(k-1)+i-1}))_{j}\wedge \xi _{j}^{p(k-1)+i-1} \\
&\leq& (B_{ij}^{-1}A_{i}\eta ^{k-1})\wedge \eta _{j}^{k-1}=\eta _{j}^{k}.
\end{eqnarray*}
\end{proof}

The correctness of the power algorithm follows from the next result.

\begin{theorem}
\label{theo-time} The sequence $\eta ^{k}$ produced by the power algorithm
initialized with $\eta ^{0}=u$ is non-increasing and converges to $P_{V}(u)$%
. Moreover, if $u$ has finite integer entries, if $V$ contains at least one
finite vector, and if all the entries of the matrices $A,B$ belong to $%
\mathbb{Z}\cup \{-\infty \}$, then, $\eta ^{m}=P_{V}(u)$ for all $m\geq
n\times d(x,V)$. %
\end{theorem}

\begin{proof}
By (\ref{power}) and (\ref{san}), we have $\eta ^{k}\geq \eta ^{k+1}\geq
P_{V}(u)\enspace(k=0,1,\ldots ).$ Hence the non-increasing sequence $\eta
^{k}$ must have a limit, $\eta ^{\infty }$, such that $u\geq \eta ^{\infty
}\geq P_{V}(u)$. To show that the equality $\eta ^{\infty }=P_{V}(u)$ holds,
by the definition of $P_{V}$ it suffices to show that $\eta ^{\infty }\in V$%
. But, passing to the limit for $k\rightarrow \infty $ in (\ref{power}) we
obtain 
\begin{equation*}
\eta ^{\infty }=B^{\sharp }A\eta ^{\infty }\wedge \eta ^{\infty },
\end{equation*}
whence by (\ref{resi4}), it follows that $A\eta ^{\infty }\geq B\eta
^{\infty },$ that is, $\eta ^{\infty }\in V.$

Assume now that the conditions of the second part of the theorem hold, and
let $v$ denote a finite vector in $V$. Then, $u\geq v\lambda $, for some
finite scalar $\lambda $, and so $P_{V}(u)\geq v\lambda $ is finite.

Moreover, we already showed that $P_{V}(u)=\eta ^{\infty }$ is the limit of
the sequence of vectors $\eta ^{k}$, and it follows from the construction of
this sequence in~\eqref{power} that for all $k$, the entries of $\eta ^{k}$
belong to $\mathbb{Z}\cup \{-\infty \}$, as soon as the entries of $A,B$ and 
$u$ do. Therefore, $P_{V}(u)\in (\mathbb{Z}\cup \{-\infty \})^{n}$, and
since we observed that $P_{V}(u)$ is finite, we must have $P_{V}(u)\in 
\mathbb{Z}^{n}$. Moreover, $\eta ^{k}\in \mathbb{Z}^{n}$ since $\eta
^{k}\geq P_{V}(u)$.

We claim that 
\begin{equation}
P_{V}(u)\backslash u=0\enspace.  \label{e-newclaim}
\end{equation}
Indeed, the inequality $P_{V}(u)\backslash u\geq 0$ follows from $%
P_{V}(u)\leq u$. If we had $P_{V}(u)\backslash u>0$, then, we would have $%
P_{V}(u)\lambda \leq u$ for some $\lambda >0$, but then the vector $%
w:=P_{V}(u)\lambda >P_{V}(u)$ would be such that $w\in V$ and $w\leq u$,
contradicting the definition of $P_{V}(u)$ as the maximal element with the
latter properties. This proves~\eqref{e-newclaim}.

Hence, 
\begin{equation*}
d(u,V)=d(u,P_{V}(u))=\big((u\backslash P_{V}(u))(P_{V}(u)\backslash u)\big)%
^{-}=(u\backslash P_{V}(u))^{-}\enspace.
\end{equation*}
Since $u$ and $P_{V}(u)$ are finite vectors, $u\backslash P_{V}(u)$ is
finite, and so, using~\eqref{oppo}, we deduce from 
\begin{equation*}
u(u\backslash P_{V}(u))\leq P_{V}(u)
\end{equation*}
that 
\begin{equation*}
u\leq P_{V}(u)(u\backslash P_{V}(u))^{-}\enspace.
\end{equation*}
Hence, 
\begin{equation*}
P_{V}(u)\leq u\leq P_{V}(u)d(u,V)\enspace.
\end{equation*}
In order to analyze the complexity of the algorithm, we return to the usual
notation for the addition, and consider the function from $\mathbb{Z}^{n}$
to $\mathbb{Z}$, 
\begin{equation*}
E(\eta ):=\sum_{i\in \lbrack n]}\big(\eta _{i}-(P_{V}(u))_{i}\big).
\end{equation*}
Observe that $E(\eta )\geq 0$ for all $\eta \geq P_{V}(u)$. Moreover, if $%
\eta ^{m}=\eta ^{m+1}$, then, $\eta ^{k}=\eta ^{m}$ must hold for all $k\geq
m$, and so, $\eta ^{m}=\lim_{k}\eta ^{k}=P_{V}(u)$. In addition, if $m$ is
the smallest index such that $\eta ^{m}=\eta ^{m+1}$, then, the sequence of
integer vectors $\eta ^{0},\ldots ,\eta ^{m}$ is strictly decreasing. In
particular, at every step $k<m$, there it as least one coordinate $i\in
\lbrack n]$ such that $\eta _{i}^{k}>\eta _{i}^{k+1}$. Thus, 
\begin{equation*}
n\times d(u,V)\geq E(\eta ^{0})>E(\eta ^{1})>\cdots >E(\eta ^{m})=0\enspace.
\end{equation*}
Since $E(\eta ^{0}),\ldots,E(\eta ^{m})$ are integers, we
deduce that $m\leq n\times d(u,V)$. %
\end{proof}

\begin{corollary}
The intersection $V=V_{1}\cap \cdots \cap V_{p}$ is not reduced to the $%
-\infty $ vector if and only if the power algorithm, initialized by taking $%
\eta ^{0}$ to be any finite vector $u$, converges to a non-$(-\infty )$
vector $\eta ^{\infty }$ (and then this vector is precisely $\eta ^{\infty
}=P_{V}(u)\in V$).\qed
\end{corollary}

The power algorithm (\ref{power}) can be rewritten componentwise as 
\begin{equation}
\eta _{i}^{k+1}=(B_{i}^{\sharp }(A\eta ^{k}))\wedge \eta _{i}^{k},\qquad
\forall i\in \lbrack n],\quad \forall k\geq 0\enspace.  \label{power2}
\end{equation}
This should be compared with the cyclic projection algorithm for $Ax\geq Bx$%
, that is, (\ref{ineq6}). Note that one step of the power algorithm requires 
$O(m)$ operations, where $m$ is the total number of finite entries in the
matrices $A$ and $B$, whereas step $i$ of the cyclic projection algorithm
only requires $O(m_{i})$ operations, where $m_{i}$ is the total number of
finite entries of the rows $A_{i}$ and $B_{i}$. Since $\xi ^{k}$ and $\eta
^{k}$ decrease to the same limit, $P_{V}(u)$, the ``sandwich'' theorem \ref
{tsandw} shows that the cyclic projection algorithm is always at least as
fast as the power algorithm, since for the same effort of computation, it
produces a closer upper bound of $P_{V}(u)$. Indeed, computing $\xi ^{pk}$
requires an $O(k(m_{1}+\cdots +m_{p}))=O(km)$ time, and computing $\eta ^{k}$
also requires an $O(km)$ time.

\begin{example}
The following example shows that the cyclic projection algorithm may yield a
speedup by a factor $n$, by comparison with the power algorithm.

Consider the system of $n-1$ inequations in $n$ variables: 
\begin{equation*}
x_{1}\leq -1+x_{n},\;x_{2}\leq -1+x_{1},\ldots ,x_{n-1}\leq -1+x_{n-2}%
\enspace.
\end{equation*}
When $n=6$, this corresponds to the following $5\times 6$ matrices 
\begin{equation*}
B=\left( 
\begin{array}{cccccc}
0 & \cdot  & \cdot  & \cdot  & \cdot  & \cdot  \\ 
\cdot  & 0 & \cdot  & \cdot  & \cdot  & \cdot  \\ 
\cdot  & \cdot  & 0 & \cdot  & \cdot  & \cdot  \\ 
\cdot  & \cdot  & \cdot  & 0 & \cdot  & \cdot  \\ 
\cdot  & \cdot  & \cdot  & \cdot  & 0 & \cdot 
\end{array}
\right) \qquad A=\left( 
\begin{array}{cccccc}
\cdot  & \cdot  & \cdot  & \cdot  & \cdot  & -1 \\ 
-1 & \cdot  & \cdot  & \cdot  & \cdot  & \cdot  \\ 
\cdot  & -1 & \cdot  & \cdot  & \cdot  & \cdot  \\ 
\cdot  & \cdot  & -1 & \cdot  & \cdot  & \cdot  \\ 
\cdot  & \cdot  & \cdot  & -1 & \cdot  & \cdot 
\end{array}
\right) \enspace,
\end{equation*}
where $-\infty $ is represented by the ``$\cdot $'' symbol.

The cyclic projection algorithm, initialized with the zero vector, yields
the sequence 
\begin{eqnarray*}
\xi ^{0}&=&(0,\ldots ,0)^{T} \\
\xi ^{1}&=&(-1,0,0,\ldots ,0)^{T} \\
\xi ^{2}&=&(-1,-2,0,\ldots ,0)^{T} \\
&\vdots&  \\
\xi ^{n-1}&=&\xi ^{n}=(-1,-2,\ldots ,-(n-1),0)^{T}\enspace.
\end{eqnarray*}

The algorithm converges in $n$ steps, and every step takes $O(1)$
operations, which makes a total of $O(n)$ operations. Indeed, note that
every row $B_{j}$ and every row $A_{j}$ have $O(1)$ entry equal to $-\infty $%
, which implies that the update of $\xi $ can be done in only $O(1)$ time.

The power algorithm, initialized with the same vector, yields the sequence 
\begin{eqnarray*}
\eta ^{0}&=&(0,\ldots ,0)^{T} \\
\eta ^{1}&=&(-1,-1,\ldots ,-1,0)^{T} \\
\eta ^{2}&=&(-1,-2,-2,\ldots ,-2,0)^{T} \\
\eta ^{3}&=&(-1,-2,-3,-3,\ldots ,-3,0)^{T} \\
&\vdots&  \\
\eta ^{n-1}&=&\eta ^{n}=(-1,-2,\ldots ,-(n-1),0)^{T}\enspace.
\end{eqnarray*}
The algorithm also converges in $n$ steps, but every step now takes an $O(n)$
time, since computing every coordinate of $B\backslash (A\eta )$ requires a $%
O(1)$ time. Thus, the power algorithm requires a total of $O(n^{2})$
operations, and the cyclic projection algorithm shows a speedup of $n$.

In this example, the matrices are very sparse. One readily gets an example
of full matrices with the same speedup by replacing every $-\infty $ entry
by a value close enough to $-\infty $, which will not modify the sequences
produced by the cyclic projection and by the power algorithm. Now, every
step of the cyclic projection algorithm takes a $O(n)$ time, and every step
of the power algorithm requires a $O(n^{2})$ time. Hence, we keep a speedup
of $n$.
\end{example}

\begin{remark}
Theorem~\ref{theo-time} gives a bound for the convergence time of the power
and cyclic projection algorithms which is pseudo-polynomial, meaning that
the convergence time is bounded by a polynomial expression in the integers
constituting the input of the problem. To see this, let us recall the
explicit expression of the projector, from~\cite{CGQ}, 
\begin{equation*}
P_{V}(u)=\sup_{i\in I}v_{i}(v_{i}\backslash u)\enspace,
\end{equation*}
where $(v_{i})_{i\in I}$ is an arbitrary generating family of $V$. A
canonical choice of the generating family consists of representatives of the
extreme rays of $V$; then, the explicit bound in~\cite[Proposition~10]{AGK10}
shows that the finite entries of the vectors $v_{i}$, and so, $d(u,V)$, are
polynomially bounded in terms of the finite entries of the matrices $A$ and $%
B$.
\end{remark}

\begin{remark}
The following simple example shows that the convergence time of both
algorithms is actually only pseudo-polynomial. Assume that $V$ is defined by
the inequalities $x_{1}\leq \max (0,-1+x_{2})$, $x_{2}\leq x_{1}$, and let
us initialize both algorithms with $u=(k,k)^{T}$, so that $(0,0)^{T}=P_{V}(u)
$ and $d(u,P_{V}(u))=k$. Then, it can be checked that both algorithms take $k
$ iterations to converge, whereas for a polynomial time algorithm, a number
of iterations polynomial in $\log k$ would be required. Let us note in this
respect that the problem of solving systems of max-plus inequalities is
equivalent to solving mean payoff games (see~\cite{mohring,AGG}), and that
the existence of a polynomial time algorithm for mean payoff games is an
open question.
\end{remark}

\end{document}